\documentclass{amsart}
\usepackage{amssymb}
\usepackage{amsmath}
\usepackage{amsthm}
\usepackage{mathrsfs}
\usepackage{verbatim}
\usepackage{xcolor}
\newtheorem{theorem}{Theorem}[section]

\newtheorem{lemma}[theorem]{Lemma}

\newtheorem*{corollary*}{Corollary}
\theoremstyle{definition}
\newtheorem{definition}[theorem]{Definition}

\theoremstyle{remark}
\newtheorem{remark}[theorem]{Remark}
\numberwithin{equation}{section}
\title[Complex Monge-Amp\`ere flows on compact K\"ahler manifolds]{Weak Solutions to the complex Monge-Amp\`ere flows on compact K\"ahler manifolds : general measures on the right-hand side}
\author{Bowoo Kang}
\address{Division of Mathematics, Department of Mathematics and Computer Science, University of Cologne, Weyertal 86-90, 50931 Cologne, Germany}
\email{bkang@uni-koeln.de}
\begin{document}
\begin{abstract}
    We show the existence of a bounded solution to the Cauchy problem for the complex Monge-Amp\`ere flow on a  compact K\"ahler manifold, with the right-hand side of the form $dt \wedge d\mu$ where $d\mu$ is either a Monge-Amp\`ere measure with a bounded potential or dominated by a Monge-Amp\`ere measure with a H\"older continuous potential. For the second case, we also prove that for a given semi-positive big from $\theta$, the $t$-slice of the solution is locally H\"older continuous on $\rm{Amp(\theta)}$ for all $t \in (0, T)$. Next, we prove a comparison principle when $d\mu$ is dominated by a Monge-Amp\`ere measure of a bounded quasi-plurisubharmonic function, which implies the uniqueness of the solution.
\end{abstract}
\maketitle
\section{Introduction}

\indent In this paper, we aim at extending the results by Guedj, Lu and Zeriahi in \cite{GLZ20} for more general measures on the right-hand side. Let $(X, \omega_X)$ be a compact K\"ahler $n$-dimensional manifold which is normalized so that $\int_X \omega_X^n = 1$. We fix $0 < T < +\infty$ and define $X_T := (0, T) \times X$ a $(2n+1)$-dimensional manifold with the parabolic boundary $\partial_0X_T := \{0\} \times X$. Let $\{\omega_t\}_{t \in [0, T)}$ be a $C^2$-family of closed semi-positive forms such that 
\begin{align*}
    -A\omega_t \leq \Dot{\omega}_t \leq \frac{A}{t}\omega_t \quad \text{ and } \quad \Ddot{\omega}_t \leq A\omega_t
\end{align*}
for some fixed constant $A > 0.$ We assume that there exists a smooth closed semipositive $(1, 1)$-form $\theta$ whose cohomology class is big and $\theta \leq \omega_t$ for all $t \in [0, T)$. We further assume that there exists a constant $B > 0$ such that $\omega_t \leq B\theta$ holds for all $t \in [0, T)$. We consider the following family of Monge-Amp\`ere flows
\begin{align}\label{CMAF}
    dt \wedge (\omega_t+dd^cu)^n = e^{\partial_tu+F(t, x, u)}dt \wedge d\mu \text{ in } X_T,
\end{align}
where $u \in \mathcal{P}(X_T, \omega_t) \cap L^{\infty}(X_T)$ is the unknown function, $\mu$ is a given positive Borel measure on $X$ satisfying
\begin{itemize}
    \item $PSH(X, \theta) \subset L^1(X, d\mu)$,
    \item $\mu(X) = \int_X \theta^n$,
\end{itemize}
and $F(t, x, r) : [0, T) \times X \times \mathbb{R} \rightarrow \mathbb{R}$ is
\begin{itemize}
    \item continuous in $[0, T) \times X \times \mathbb{R}$,
    \item uniformly quasi-increasing in $r$, i.e. $\forall (t, x) \in [0, T) \times X$, $\exists \lambda_F \geq 0$ such that 
    \begin{align*}
        r \mapsto F(t, x, r) + \lambda_Fr \text{ is increasing in }\mathbb{R},
    \end{align*}
    \item locally uniformly Lipschitz in $(t, r)$,
    \item locally uniformly semi-convex in $(t, r)$.
\end{itemize}
As usual, $d = \partial + \overline{\partial}$ and $d^c = i(\overline{\partial} - \partial)$ so that $dd^c = i\partial \overline{\partial}$. This setup is same with \cite{GLZ20}, except for the existence of constant $B > 0$ satisfying $\omega_t \leq B\theta$. We remark that even though it is not mentioned in \cite{GLZ20}, it follows from the proof of \cite[Lemma 5.1]{GLZ20} that there exists such a constant $B$ for $\omega_t$ and $\theta$ in \cite[Section 5]{GLZ20}. \\
\indent The interest on the complex Monge-Amp\`ere flow was first motivated by its relationship with the K\"ahler-Ricci flow. In 1985, Cao \cite{Cao85} reproved the Yau's solution to the Calabi conjecture by solving the complex Monge-Amp\`ere flow on a compact K\"ahler manifold and proving its long-time existence. Since then, the K\"ahler-Ricci flow has been an important topic in itself. We refer the readers to survey papers \cite{SW13, To18} and references therein for more details. Moreover, various geometric flows have been extensively studied, including Monge-Amp\`ere flows, J-flows, inverse Monge-Amp\`ere flows and so on. Recently, Phong and T\^o \cite{PT21} introduced the notion of parabolic C-subsolution, originated from the work by Guan \cite{Gu14} and Sz\'ekelyhidi \cite{Sz18}, and provided the unified approach to the study of many geometric flows. \\
\indent For its geometrical application, the study on weak solutions of the degenerate complex Monge-Amp\`ere flows has been one of the main topic. Eyssidieux, Guedj, and Zeriahi \cite{EGZ15, EGZ16} developed a viscosity approach to study the degenerate complex Monge-Amp\`ere flows. They used this viscosity approach to study the K\"ahler-Ricci flow on minimal models of positive Kodaira dimension in \cite{EGZ18}. Subsequently in 2020, Guedj, Lu, and Zeriahi \cite{GLZ20, GLZ21-2} developed a pluripotential approach and applied it to study the K\"ahler-Ricci flow on varieties with log terminal singularities. They also proved in \cite{GLZ21-1} that viscosity and pluripotential solutions are equivalent under some conditions. \\
\indent Let us give background on the pluripotential approach of Guedj, Lu and Zeriahi, where they successfully extended the works of Bedford and Taylor in \cite{BT76, BT82} into the parabolic setting and proved the existence and uniqueness of a solution to the complex Monge-Amp\`ere flow in $(0, T) \times X$, when $\mu = g\omega_X^n$ where $g \in L^p(X)$ for some $p > 1$ and $g > 0$ a.e. in $X$. They proved in \cite[Theorem A, B]{GLZ20} that there exists a unique function $u \in \mathcal{P}(X_T, \omega_t) \cap L^{\infty}(X_T)$ such that
\begin{itemize}
    \item $u$ satisfies (\ref{CMAF}),
    \item $u$ is continuous in $(0, T) \times$ $\rm{Amp(\theta)}$,
    \item $t \mapsto u(t, x)$ is locally uniformly semi-concave in $(0, T)$,
    \item $\lim_{t \rightarrow 0+}u(t, x) = u_0(x)$ for all $x \in X$,
\end{itemize}
where $u_0 \in PSH(X, \omega_0) \cap L^{\infty}(X)$ and $\rm{Amp(\theta)}$ is an ample locus of $\theta$. We extend this result to the case when the right-hand side is a general measure that can be singular with respect to the volume form $\omega_X^n$. \\
\indent We fix some further notations and expressions which is frequently used throughout this paper. We mainly deal with a family of quasi-plurisubharmonic functions.
\begin{definition}[{\cite[Definition 1.1]{GLZ20}}]
    A parabolic potential is $u : X_T \rightarrow [-\infty, +\infty)$ such that
    \begin{itemize}
        \item $x \mapsto u(t, x)$ is $\omega_t$-plurisubharmonic on $X$ for all $t \in (0, T)$,
        \item $u$ is locally uniformly Lipschitz in $(0, T)$.
    \end{itemize}
    We denote by $\mathcal{P}(X_T, \omega_t)$ the set of all parabolic potentials.
\end{definition}
For a given sequence of Borel measures $\{\nu_j\}_{j = 1}^{\infty}$, we say for convenience that $\nu_j \rightarrow \nu$ weakly as $j \rightarrow \infty$ for some Borel measure $\nu$, when the convergence is in the weak sense of Radon measures. Let $dV$ denote the normalized smooth volume form $\omega_X^n$ and let $d\ell$ denote the Lebesgue measure on $\mathbb{R}$. \\
\indent 
Now, let us state our first main result. We prove the existence of a solution to the Cauchy problem when the measure $\mu$ on the right-hand side is either a Monge-Amp\`ere measure with a bounded potential or dominated by a Monge-Amp\`ere measure with a H\"older continuous potential. 
\begin{theorem}\label{intro_main result 1}
    Let $u_0 \in PSH(X, \omega_0) \cap L^{\infty}(X)$. We consider the following Cauchy problem.
    \begin{align}\label{Theorem 1.2, equation}
        \begin{cases}
            &u \in \mathcal{P}(X_T, \omega_t) \cap L^{\infty}(X_T), \\
            &dt \wedge (\omega_t+dd^cu)^n = e^{\partial_tu+F(t, x, u)}dt \wedge d\mu \text{ in $X_T$}, \\
            &\lim_{t \rightarrow 0+}u(t, \cdot) = u_0 \text{ in $L^1(X, d\mu)$}. 
        \end{cases}
    \end{align}
    \begin{itemize}
        \item [(i)] If the measure $\mu$ satisfies
        \begin{align*}
            d\mu = e^{c\psi}(\omega_X+dd^c\psi)^n    
        \end{align*}
        for some $c \in \mathbb{R}$ and $\psi \in PSH(X, \omega_X) \cap L^{\infty}(X)$, then there exists $u$ satisfying (\ref{Theorem 1.2, equation}).
        \item [(ii)] If the measure $\mu$ satisfies 
        \begin{align*}
            d\mu \leq C(\omega_X+dd^c\phi)^n    
        \end{align*}
        for some $C > 0$ and a H\"older continuous $\omega_X$-psh function $\phi$, then there exists $u$ satisfying (\ref{Theorem 1.2, equation}). Moreover, $u(t, \cdot)$ is locally H\"older continuous on Amp($\theta$) for all $t \in (0, T)$ and jointly continuous on $(0, T) \times \text{Amp}(\theta)$. The H\"older exponent of $u(t, \cdot)$ does not depend on $t$.
    \end{itemize}
\end{theorem}
We would like to remark that a Monge-Amp\`ere measure of a H\"older continuous quasi-plurisubharmonic function can be singular with respect to the volume form on $X$. For example, Hiep proved in \cite{Hiep10} that if $\mu$ is the volume measure of a smooth real hypersurface, then it is a Monge-Amp\`ere measure of some H\"older continuous quasi-plurisubharmonic function. Later, Vu extended this result in \cite{Vu18} when $\mu$ is a measure supported on a totally real submanifold. \\
\indent Our second main result is a general comparison principle when the measure $\mu$ is dominated by a Monge-Amp\`ere measure of a bounded quasi-plurisubharmonic function. It deals with more general measures compared to the one considered in the Theorem \ref{intro_main result 1}, and this result is a generalization to the manifold case of \cite[Theorem 1.2]{Ka25b}. The uniqueness of the solution to the Cauchy problem directly follows from this theorem.
\begin{theorem}\label{intro_main result 2}
    Let $u_0, v_0 \in PSH(X, \omega_0) \cap L^{\infty}(X)$. Assume that there exist $\varphi \in PSH(X, \theta) \cap L^{\infty}(X)$ and a constant $C > 0$ satisfying
    \begin{align*}
        d\mu \leq C(\theta+dd^c\varphi)^n \text{ in } X.
    \end{align*}
        Let $u, v \in \mathcal{P}(X_T, \omega_t) \cap L^{\infty}(X_T)$ satisfy $\lim_{t \rightarrow 0+}u(t, \cdot) = u_0$ and $\lim_{t \rightarrow 0+}v(t, \cdot) = v_0$ in $L^1(X, d\mu)$. Assume that
    \begin{itemize}
        \item [(a)] $u$ and $v$ are locally uniformly semi-concave in $(0, T)$,
        \item [(b)] $dt \wedge (\omega_t+dd^cu)^n \geq e^{\partial_tu+F(t, x, u)}dt \wedge d\mu$ in $X_T$,
        \item [(c)] $dt \wedge (\omega_t+dd^cv)^n \leq e^{\partial_tv+F(t, x, v)}dt \wedge d\mu$ in $X_T$.
    \end{itemize}
    If $u_0 \leq v_0$, then $u \leq v$.
\end{theorem}
    \mbox{}\\
\indent \textbf{Organization.} In Section $2$, we first recall some definitions and lemmas of the parabolic pluripotential theory from \cite{GLZ20}. We also prove some useful lemmas related to the parabolic Monge-Amp\`ere operator. In Section $3$, we prove some lemmas about the convergence of the right-hand side of the Monge-Amp\`ere flow. In Section $4$, we give a proof of Theorem \ref{intro_main result 1}. In Section $5$, we give a proof of Theorem \ref{intro_main result 2}.\\
\mbox{}\\
\indent \textbf{Acknowledgements.}
I would like to thank my advisor, Ngoc Cuong Nguyen,
for suggesting this problem and providing invaluable guidance and support. I also
would like to thank S{\l}awomir Ko{\l}odziej for his valuable comments.  This work was supported by the
National Research Foundation of Korea (NRF) funded by the Korean government (MSIT) RS-2026-25470686. We would like to thank the anonymous referee
for reading the manuscript carefully and giving many valuable suggestions.
\section{Preliminaries}
We briefly introduce basic definitions and lemmas related to the parabolic potentials and parabolic Monge-Amp\`ere operator. We also explain some estimates of the solutions obtained in \cite{GLZ20}. We refer the reader to \cite{GLZ20} and \cite{GLZ21-2} for detailed proofs. 
\subsection{Parabolic Monge-Amp\`ere operator}
We first recall a parabolic version of the Chern-Levine-Nirenberg inequality (see e.g. \cite[Theorem 3.9]{GZ17} for the elliptic case).
\begin{lemma}[{\cite[Lemma 1.9]{GLZ20}}]\label{CLN}
    Fix $u \in \mathcal{P}(X_T, \omega_t) \cap L^{\infty}_{loc}(X_T)$ and $\chi(t, x)$ a continuous test function in $X_T$. The map 
    \begin{align*}
        \Gamma_{\chi} : t \mapsto \int_{X}\chi(t, \cdot)(\omega_t+dd^cu)^n
    \end{align*}
    is continuous in $(0, T)$ and bounded, with
    \begin{align*}
        \sup_{0 < t < T}\left\vert \int_X \chi(t, \cdot)(\omega_t+dd^cu(t, \cdot))^n\right\vert \leq (\sup_{X_T}\left\vert \chi\right\vert)B^n\int_X \theta^n.
    \end{align*}
\end{lemma}
By using the above result, we define a parabolic Monge-Amp\`ere operator. 
\begin{definition}[{\cite[Definition 1.10]{GLZ20}}]
    Let $u \in \mathcal{P}(X_T, \omega_t) \cap L^{\infty}_{loc}(X_T)$. The map
    \begin{align*}
        C_c(X_T) \ni \chi \mapsto \int_{X_T}\chi dt \wedge (\omega_t + dd^cu)^n := \int_0^T dt \left(\int_X \chi(t, \cdot)(\omega_t+dd^cu)^n\right)
    \end{align*}
    defines a positive Radon measure on $X_T$, which is denoted by $dt \wedge (\omega_t+dd^cu)^n$.
\end{definition}
The continuity of this parabolic Monge-Amp\`ere operator along a sequence of parabolic potentials follows from the continuity of corresponding elliptic Monge-Amp\`ere operators.
\begin{lemma} \label{Lemma 2.3}
    Fix $u \in \mathcal{P}(X_T, \omega_t) \cap L^{\infty}_{loc}(X_T)$ and let $\{u_j\}_{j = 1}^{\infty}$ be a sequence of functions in $\mathcal{P}(X_T, \omega_t) \cap L^{\infty}_{loc}(X_T)$ such that for almost every $t \in (0, T)$,
    \begin{align*}
        (\omega_t+dd^cu_j(t, \cdot))^n \rightarrow (\omega_t+dd^cu(t, \cdot))^n \text{ weakly as } j \rightarrow \infty.
    \end{align*}
    Then
    \begin{align*}
        dt \wedge (\omega_t+dd^cu_j)^n \rightarrow dt \wedge (\omega_t+dd^cu)^n \text{ weakly as } j \rightarrow \infty.
    \end{align*}
\end{lemma}
\begin{proof}
    It follows from the proof of \cite[Proposition 1.11]{GLZ20}.
\end{proof}
Let us recall the capacity in \cite{Ko03} which is modeled on \cite[Definition 3.1]{BT82}.
\begin{definition}
    For a Borel subset $E \subset X$,
    \begin{align*}
        cap_{\theta}(E) = \sup\left\{\int_E(\theta+dd^c\psi)^n ~\mid~ \psi \in PSH(X, \theta), -1\leq \psi \leq 0\right\}.
    \end{align*}
\end{definition}
\begin{definition}
    A sequence of Borel functions $\{f_j\}_{j = 1}^{\infty}$ converges in $cap_{\theta}$ to a Borel function $f$ in $X$ if for all $\delta > 0$, 
    \begin{align*}
        \lim_{j \rightarrow \infty}cap_{\theta}(\{\lvert f_j - f\rvert \geq \delta\}) = 0.
    \end{align*}
\end{definition}
The following result is useful that it will be used several times below. It is the consequence of \cite[Theorem 1]{Xi08} and Lemma \ref{Lemma 2.3}.
\begin{lemma}
    Fix $u \in \mathcal{P}(X_T, \omega_t) \cap L^{\infty}_{loc}(X_T)$ and let $\{u_j\}_{j = 1}^{\infty}$ be a sequence of functions in $\mathcal{P}(X_T, \omega_t) \cap L^{\infty}_{loc}(X_T)$ such that for almost every $t \in (0, T)$,
    \begin{align*}
        u_j(t, \cdot) \rightarrow u(t, \cdot) \text{ in $cap_{\omega_t}$ as } j \rightarrow \infty.
    \end{align*}
    Then
    \begin{align*}
        dt \wedge (\omega_t+dd^cu_j)^n \rightarrow dt \wedge (\omega_t+dd^cu)^n \text{ weakly as } j \rightarrow \infty.
    \end{align*}
\end{lemma}
\subsection{Time derivatives of parabolic potentials} Because of the local uniform Lipschitzness, time derivatives of parabolic potentials are well-defined a.e. We explain this in more details which is a slight improvement of \cite[Lemma 1.6]{GLZ20}.
\begin{lemma}\label{Lemma 2.7}
    Let $u \in \mathcal{P}(X_T, \omega_t)$. Then there exists a Borel set $E \subset X_T$ such that $\partial_tu(t, x)$ exists for all $(t, x) \in X_T \setminus E$ and
    \begin{align*}
        \int_{E \cap ((0, T) \times K)}dt \wedge (\theta+dd^c\varphi)^n = 0
    \end{align*}
    for every compact set $K \subset X$ and for every $\varphi \in PSH(X, \theta)$ satisfying $-1 \leq \varphi \leq 0$.
\end{lemma}
\begin{proof}
    The proof is same with the local case in \cite[Lemma 2.7]{Ka25}.
\end{proof}
We now introduce the local uniform semi-concavity, which is used in \cite[Theorem 1.14]{GLZ20} to obtain the convergence of time derivatives.
\begin{definition}
    \mbox{} 
    \begin{itemize}
        \item [(a)] We say that $u : X_T \rightarrow \mathbb{R}$ is locally uniformly semi-concave in $(0, T)$ if for a given compact set $J \Subset (0, T)$, there exists $C = C(J) > 0$ such that for all $x \in X$,
        \begin{align}\label{2.8.1}
            t \mapsto u(t, x) - Ct^2 \text{ is concave in } J.
        \end{align}
        \item [(b)] We say that the family $\mathcal{A}$ of functions mapping from $X_T$ to $\mathbb{R}$ is locally uniformly semi-concave in $(0, T)$ if for a given compact $J \Subset (0, T)$, there exists $C = C(J) > 0$ such that (\ref{2.8.1}) holds for all $x \in X$ and $u \in \mathcal{A}$.
    \end{itemize}
\end{definition}
The assumption on local uniform semi-concavity helps to prove the convergence of time derivatives, which is a slight improvement of \cite[Theorem 1.14]{GLZ20}.
\begin{lemma}\label{Lemma 2.9}
    Let $\{u_j\}_{j = 1}^{\infty}$ be a sequence of functions in $\mathcal{P}(X_T, \omega_t)$. Assume that
    \begin{itemize}
        \item[(a)] for a.e. $t \in (0, T)$, $u_j(t, \cdot) \rightarrow u(t, \cdot)$ in $cap_{\omega_t}$ as $j \rightarrow \infty$,
        \item[(b)] $\{u_j\}_{j = 1}^{\infty}$ is locally uniformly semi-concave in $(0, T)$.
    \end{itemize}
    Then there exists a Borel subset $E \subset X_T$ such that for all compact subsets $K \subset X$, $\lim_{j \rightarrow \infty}\partial_tu_j = \partial_tu$ in $((0, T) \times K) \setminus E$. Moreover,
    \begin{align*}
        \int_{E \cap ((0, T) \times K)}dt \wedge (\theta + dd^c\varphi)^n = 0.
    \end{align*}
    for all $\varphi \in PSH(X, \theta)$ satisfying $-1 \leq \varphi \leq 0$,
\end{lemma}
\begin{proof}
    Since $\theta \leq \omega_t$ for all $t \in [0, T)$, $u_j(t, \cdot) \rightarrow u(t, \cdot)$ in $cap_{\theta}$ for a.e. $t \in (0, T)$ as $j \rightarrow \infty$. With this observation, the proof is similar to the local case in \cite[Lemma 2.9]{Ka25}.
\end{proof}
\subsection{Estimates on the solutions}
In this section, we assume that there exists $\psi \in PSH(X, \theta) \cap C^{\infty}(X)$ satisfying
\begin{align*}
    \begin{cases}
        &(\theta+dd^c\psi)^n = d\mu \text{ in $X$}, \\
        &\sup_{X}\psi = 0.
    \end{cases}
\end{align*}
We further assume in this section that $d\mu \geq \varepsilon dV$ for some $\varepsilon > 0$.
It follows from \cite[Theorem 3.4]{GLZ20} that there exists $u \in \mathcal{P}(X_T, \omega_t) \cap L^{\infty}(X_T)$ satisfying
\begin{align*}
    \begin{cases}
        &dt \wedge (\omega_t+dd^cu)^n = e^{\partial_tu+F(t, x, u)}dt \wedge d\mu \text{ in $X_T$,}\\
        &\lim_{t \rightarrow 0+}u(t, \cdot) = u_0 \text{ in }L^1(X, dV).
    \end{cases}
\end{align*}
\begin{lemma}\label{uniform estimate}
    There exists a constant $M > 0$ such that
    \begin{align}\label{uniform bound}
        \lvert u\rvert \leq M
    \end{align}
    for all $(t, x) \in X_T$. The constant $M$ depends only on $\lambda_F$, $\sup_{X_T}\lvert F(t, x, 0)\rvert$, $\lVert \psi\rVert_{L^{\infty}(X)}$, $\lVert u_0\rVert_{L^{\infty}(X)}$, and $B$.
\end{lemma}
\begin{proof}
   We have $\omega_t \leq B\theta$ for all $t \in (0, T)$ and
    \begin{align*}
        (B\theta + dd^c(B\psi))^n = B^nd\mu = e^{n\log B}d\mu \text{ in $X$.}
    \end{align*}
    Hence the bound on $u$ follows from the proof of \cite{GLZ20}. Indeed, we first assume that $F$, $u_0$ and $\omega_t$ are smooth. By following the proof of \cite[Proposition 2.1]{GLZ20}, we get (\ref{uniform bound}). Next, we approximate $F$, $u_0$ and $\omega_t$ as in the proof of \cite[Theorem 3.4]{GLZ20}. Therefore we get (\ref{uniform bound}).
\end{proof}
\begin{lemma}\label{lipschitzness}
    There exists a constant $\kappa > 0$ such that
    \begin{align*}
        \left\vert \partial_tu \right\vert \leq \frac{\kappa}{t}
    \end{align*}
    for all $(t, x) \in X_T$ where $\partial_tu(t, x)$ is well-defined. The constant $\kappa$ depends only on $T$, $\lVert \partial F/\partial r\rVert_{L^{\infty}}$, $\lVert \partial F/ \partial t\rVert_{L^{\infty}}$, $\lVert \psi\rVert_{L^{\infty}(X)}$, and $\lVert u\rVert_{L^{\infty}(X_T)}$. Here the norms $\lVert \partial F/ \partial t\rVert_{L^{\infty}}$ and $\lVert \psi\rVert_{L^{\infty}(X)}$ are computed on $X_T \times [-\lVert u\rVert_{L^{\infty}(X_T)}, \lVert u\rVert_{L^{\infty}(X_T)}]$.
\end{lemma}
\begin{proof}
    It follows from the proof of \cite[Theorem 2.5]{GLZ20} and \cite[Theorem 3.4]{GLZ20}.
\end{proof}
\begin{lemma}\label{concavity1}
    Let $J \Subset (0, T)$ be an open interval. There exists a constant $C  > 0$ such that
    \begin{align*}
        t \mapsto u(t, x) + C \log t \text{ is concave in } J
    \end{align*}
    for all $x \in X$. The constant $C$ depends only on $A$, $T$, $\lVert \partial F/\partial r\rVert_{L^{\infty}}$, $\lVert \partial F/\partial t\rVert_{L^{\infty}}$, $\lVert \partial^2F/\partial r^2\rVert_{L^{\infty}}$, $\lVert \partial^2F/\partial t^2\rVert_{L^{\infty}}$, $\lVert \psi\rVert_{L^{\infty}(X)}$, and $\lVert u\rVert_{L^{\infty}(X_T)}$. Here the norms of the partial derivatives of $F$ are computed on $X_T \times [-\lVert u\rVert_{L^{\infty}(X_T)}, \lVert u\rVert_{L^{\infty}(X_T)}]$.
\end{lemma}
\begin{proof}
    It follows from the proof of \cite[Theorem 2.9]{GLZ20} and \cite[Theorem 3.4]{GLZ20}.
\end{proof}
\subsection{Parabolic comparison principle}
In this section, we prove lemmas which will be used to prove Theorem \ref{intro_main result 2}. They are mainly about computing and comparing some sublevel sets with respect to Monge-Amp\`ere measures. The first one generalizes Lemma \ref{CLN}. 
\begin{lemma}\label{borel measurable}
    Let $u \in \mathcal{P}(X_T, \omega_t) \cap L^{\infty}_{loc}(X_T)$ and $E \subset X_T$ be a Borel subset. The map
    \begin{align*}
        \Gamma_{E} : t \mapsto \int_{E_t}(\omega_t+dd^cu(t, \cdot))^n
    \end{align*}
    is Borel measurable in $(0, T)$, where $E_t := \{x \in X ~\mid~ (t, x) \in E\}$.
\end{lemma}
\begin{proof}
    The proof is same with the local case in  \cite[Lemma 2.2]{Ka25b}.
\end{proof}
We also need the following lemma which will be several times.
\begin{lemma}\label{fubini type lemma}
    Let $u \in \mathcal{P}(X_T, \omega_t) \cap L^{\infty}_{loc}(X_T)$. For any Borel subset $E \subset X_T$, we have
    \begin{align*}
        \int_{E}dt \wedge (\omega_t+dd^cu)^n = \int_0^T dt \int_{E_t}(\omega_t+dd^cu(t, \cdot))^n,
    \end{align*}
    where $E_t := \{x \in X ~\mid~ (t, x) \in E\}$.
\end{lemma}
\begin{proof}
    The proof is same with the local case in \cite[Lemma 2.3]{Ka25b}.
\end{proof}
Now it is ready to prove a comparison principle and a domination principle with respect to a parabolic Monge-Amp\`ere operator. They are counterparts of the elliptic Monge-Amp\`ere operator (see e.g. \cite[Proposition 10.6, Proposition 10.11]{GZ17}).
\begin{lemma}
    Assume that $u, v \in \mathcal{P}(X_T, \omega_t) \cap L^{\infty}(X_T)$. Then
    \begin{align*}
        \int_{\{v < u\}}dt \wedge (\omega_t+dd^cu)^n \leq \int_{\{v < u\}}dt \wedge (\omega_t+dd^cv)^n.
    \end{align*}
\end{lemma}
\begin{proof}
    It follows from the classical comparison principle (see e.g. \cite[Proposition 9.2]{GZ17}) that for all $t \in (0, T)$, 
    \begin{align}\label{2.3}
        \int_{\{v(t, \cdot) < u(t, \cdot)\}}(\omega_t+dd^cu(t, \cdot))^n \leq \int_{\{v(t, \cdot) < u(t, \cdot)\}}(\omega_t+dd^cv(t, \cdot))^n.
    \end{align}
    It follows from Lemma \ref{borel measurable} that each side of (\ref{2.3}) is a Borel measurable function of $t \in (0, T)$. By integrating each side with respect to $dt$, we get
    \begin{align*}
        \int_{0}^T dt \int_{\{v(t, \cdot) < u(t, \cdot)\}}(\omega_t+dd^cu(t, \cdot))^n \leq \int_{0}^T dt \int_{\{v(t, \cdot) < u(t, \cdot)\}}(\omega_t+dd^cv(t, \cdot))^n.
    \end{align*}
    By using Lemma \ref{fubini type lemma}, we get the conclusion.
\end{proof}
\begin{lemma}\label{domination principle}
    Assume that $u, v \in \mathcal{P}(X_T, \omega_t) \cap L^{\infty}(X_T)$ satisfy 
    \begin{align*}
        \int_{\{v < u\}}dt \wedge (\omega_t+dd^cv)^n = 0.
    \end{align*}
    Then $u \leq v$ in $X_T$.
\end{lemma}
\begin{proof}
    Let us denote $E := \{v < u\}$. It follows from Lemma \ref{fubini type lemma} that
    \begin{align*}
        \int_{E}dt \wedge (\omega_t+dd^cv)^n = \int_0^T dt \int_{E_t}(\omega_t+dd^cv)^n = 0,
    \end{align*}
    where $E_t := \{x \in X ~\mid~ (t, x ) \in E\}$. Therefore we have
    \begin{align*}
        \int_{E_t}(\omega_t+dd^cv(t, \cdot))^n = 0
    \end{align*}
    for a.e. $t \in (0, T)$. It follows from the classical domination principle (see e.g. \cite[Proposition 10.11]{GZ17}) that 
    \begin{align*}
        u(t, \cdot) \leq v(t, \cdot)
    \end{align*}
    holds in $X$ for a.e. $t \in (0, T)$, which implies that $u \leq v$ everywhere in $X_T$.
\end{proof}
\section{Convergence Lemmas}
Let $\{u_j\}_{j = 1}^{\infty}$ be a sequence of uniformly bounded parabolic potentials and let $\{\mu_j\}_{j = 1}^{\infty}$ be a sequence of positive Borel measures in $X$. Assume that $d\mu_j \rightarrow d\mu$ weakly as $j \rightarrow \infty$. In this section, we find a sufficient condition for the weak convergence of
\begin{align*}
    e^{\partial_tu_j+F(t, x, u_j)}dt \wedge d\mu_j \rightarrow e^{\partial_tu+F(t, x, u)}dt \wedge d\mu \text{ as } j \rightarrow \infty.
\end{align*}
The result below is analogous to \cite[Lemma 3.3]{Ka25}.
\begin{lemma}\label{integration by parts}
    Let $\psi \in PSH(X, \theta) \cap L^{\infty}(X)$ and $u \in \mathcal{P}(X_T, \omega_t) \cap L^{\infty}(X_T)$. Then for every smooth test function $\chi(t, x)$ on $X_T$, 
    \begin{align*}
        \int_{X_T}\chi \partial_tudt \wedge (\theta + dd^c\psi)^n = -\int_{X_T}u \partial_t\chi dt \wedge (\theta + dd^c\psi)^n.
    \end{align*}
\end{lemma}
\begin{proof}
    The proof is same with \cite[Lemma 3.1]{Ka25}. We provide the details for the reader's convenience. Fix $\chi(t, x)$ a smooth test function on $X_T$. Let
    \begin{align*}
        S(t) := \int_{X}\chi(t, \cdot)u(t, \cdot)(\theta+dd^c\psi)^n.
    \end{align*}
    Assume that $supp(\chi) \subset J \times K \Subset X_T$, where $J$ and $K$ are compact subsets of $(0, T)$ and $X$ respectively. Since $u \in \mathcal{P}(X_T, \omega_t)$, there exists a constant $\kappa > 0$ such that
    \begin{align*}
        u(t, x) \leq u(s, x) + \kappa \lvert t-s\rvert
    \end{align*}
    for all $t, s \in J$ and $x \in X$. Thus for every $t_0, t_1 \in J$,
    \begin{align*}
        \lvert S(t_0) - S(t_1) \rvert &= \left\vert \int_K (\chi(t_0, \cdot)u(t_0, \cdot) - \chi(t_1, \cdot)u(t_1, \cdot))(\theta+dd^c\psi)^n\right\vert \\
        &\leq \int_K \lvert \chi(t_0, \cdot) - \chi(t_1, \cdot)\rvert \lvert u(t_0, \cdot)\rvert (\theta + dd^c\psi)^n \\
        &\quad + \int_K \lvert \chi(t_1, \cdot)\rvert \lvert u(t_0, \cdot) - u(t_1, \cdot)\rvert (\theta+dd^c\psi)^n \\
        &\leq C\lvert t-t_0\rvert
    \end{align*}
    where 
    \begin{align*}
        C = \left(\int_K (\theta+dd^c\psi)^n\right)\left(\left\Vert \partial_t\chi \right\Vert_{L^{\infty}(J \times K)}\left\Vert u\right\Vert_{L^{\infty}(X_T)}+\lVert \chi\rVert_{L^{\infty}(X_T)}\kappa \right) > 0.
    \end{align*}
    This implies that $S$ is Lipschitz. Therefore,
    \begin{align*}
        \int_{X_T}\chi \partial_tudt \wedge (\theta + dd^c\psi)^n &= \int_0^T dt \int_{X}\chi(t, \cdot)\partial_tu(t, \cdot)(\theta+dd^c\psi)^n \\
        &=\int_0^T S'(t)dt - \int_{0}^Tdt\int_{X}\partial_t\chi(t, \cdot)u(t, \cdot)(\theta+dd^c\psi)^n \\
        &= -\int_{X_T}u\partial_t\chi dt \wedge (\theta+dd^c\psi)^n.
    \end{align*}
    Here we used the integration by parts for the second equality. The third equality holds by the Lipschitzness of $S$.
\end{proof}
The solutions of the flow with degenerate data is often obtained from approximating sequences. Therefore, the convergence of the each side of the equation of the flow is crucial. The following technical lemma provides a sufficient condition for the convergence of the right-hand side, and it plays an essential role in the proof of Theorem \ref{intro_main result 1}. This is an improvement of \cite[Theorem 1.14]{GLZ20} and analogous to \cite[Lemma 3.3]{Ka25} in local setting, with additional terms $e^{-c\psi_j}$ and $e^{-c\psi}$.
\begin{lemma}\label{convergence}
    Assume that $\psi_j, \psi \in PSH(X, \theta) \cap L^{\infty}(X)$ satisfy
    \begin{itemize}
        \item [(a)] $\psi_j \rightarrow \psi$ in $cap_{\theta}$ as $j \rightarrow \infty$, 
        \item [(b)] $\{\psi_j\}_{j = 1}^{\infty}$ is uniformly bounded.
    \end{itemize}
    Assume that $u_j, u \in \mathcal{P}(X_T, \omega_t) \cap L^{\infty}(X_T)$ satisfy
    \begin{itemize}
        \item [(c)] for almost every $t \in (0, T)$, $u_j(t, \cdot) \rightarrow u(t, \cdot)$ in $cap_{\omega_t}$ as $j \rightarrow \infty$,
        \item [(d)] $\{u_j\}_{j = 1}^{\infty}$ is locally uniformly bounded,
        \item [(e)] there exists a constant $\kappa_0$ satisfying
        \begin{align*}
             \left\vert \partial_tu_j(t, x) \right\vert \leq \frac{\kappa_0}{t} \text{ for all } j
        \end{align*}
        for all $(t, x) \in X_T$ where $\partial_tu_j$ is well-defined,
        \item [(f)] for each open interval $J \Subset (0, T)$, there exists a constant $C_0 = C_0(J) > 0$ satisfying
        \begin{align*}
            t \mapsto u_j(t, x) + C_0\log t \text{ is concave in } J 
        \end{align*}
        for all $x \in X$ and $j$.
    \end{itemize}
    Then for every $c \in \mathbb{R}$,
    \begin{align*}
        e^{\partial_tu_j+F(t, x, u_j)-c\psi_j}dt \wedge (\theta+dd^c\psi_j)^n \rightarrow e^{\partial_tu+F(t, x ,u)-c\psi}dt \wedge (\theta+dd^c\psi)^n \text{ weakly }
    \end{align*}
    as $j \rightarrow \infty$.
\end{lemma}
\begin{proof}
    Let $M_\psi > 0$ be a constant satisfying $\lvert \psi_j\rvert \leq M_\psi$ in $X$ for all $j$. Fix $\chi(t, x)$ a positive test function in $X_T$ and assume that $supp(\chi) \Subset J \times K \Subset X_T$. Let $M_{u} > 0$ be a constant satisfying $\lvert u_j \rvert \leq M_u$ in $J \times K$ for all $j$. Let us denote 
    \begin{align*}
        M_F := \lVert F(t, x, M_u)\rVert_{L^{\infty}(J \times K)} > 0.
    \end{align*} 
    By Lemma \ref{Lemma 2.7}, there exists a sequence $\{E_j\}_{j = 1}^{\infty}$ of Borel subsets of $X_T$ such that for each $j$, $\partial_tu_j(t, x)$ exists for all $(t, x) \in X_T \setminus E_j$ and
    \begin{align*}
        \int_{E_j \cap ((0, T) \times K)}dt \wedge (\theta+dd^c\varphi)^n = 0
    \end{align*}
    for all $\varphi \in PSH(X, \theta)$ satisfying $-1 \leq \varphi \leq 0$. Let us define $E = \bigcup_{j = 1}^{\infty}E_j$. Then $E \subset X_T$ is a Borel subset such that for all $j$,
    \begin{align}\label{E}
        \int_{E \cap ((0, T) \times K)}dt \wedge (\theta + dd^c\psi_j)^n = \int_{E \cap ((0, T) \times K)}dt \wedge (\theta + dd^c\psi)^n = 0
    \end{align}
    and $\partial_tu_j(t, x)$ exists for all $(t, x) \in X_T \setminus E$. Hence we have
    \begin{align}\label{first derivative}
         \left\vert \partial_tu_j(t, x)\right\vert  \leq \frac{\kappa_0}{t}
    \end{align}
    for all $j$ and for all $(t, x) \in X_T \setminus E$. \\
    \indent Note that $u_j(t, \cdot) \rightarrow u(t, \cdot)$ in $cap_{\omega_t}$ as $j \rightarrow \infty$ for a.e. $t \in (0, T)$ and the family $\{u_j\}_{j = 1}^{\infty}$ is locally uniformly semi-concave in $(0, T)$. Thus by Lemma \ref{Lemma 2.9}, there exists a Borel subset $G \subset X_T$ such that for all $j$,
    \begin{align*}
        \int_{G \cap ((0, T) \times K)}dt \wedge (\theta + dd^c\psi_j)^n = \int_{G \cap ((0, T) \times K}dt \wedge (\theta + dd^c\psi)^n = 0
    \end{align*}
    and $\partial_tu_j$ pointwisely converges to $\partial_tu$ in $((0, T) \times K) \setminus G$ as $j \rightarrow \infty$. Hence we have $\left\vert \partial_tu(t, x)\right\vert \leq \frac{\kappa_0}{t}$ for all $(t, x) \in ((0, T) \times K)\setminus G$. \\
     \indent Next, fix an interval $J \Subset J' \Subset (0, T)$. By the assumption, there exists a constant $C_0 > 0$ such that $t \mapsto u_j(t, x) + C_0\log t$ is concave in $J'$ for all $x \in X$ and $j$. Since $u_j \rightarrow u$ weakly as $j \rightarrow \infty$, $t \mapsto u(t, x)+C_0 \log t$ is concave in $J'$ for all $x \in X$. 
    \indent Fix $\varepsilon > 0$ and $c \in \mathbb{R}$. We first show that
    \begin{align}\label{result 1}
        \left(e^{\partial_tu_j+F(t, x, u_j)-c\psi_j} - e^{\partial_tu+F(t, x, u)-c\psi}\right)dt \wedge (\theta+dd^c\psi_j)^n \rightarrow 0 \text{ weakly }
    \end{align}
    as $j \rightarrow \infty$. In fact, for a sufficiently small $0 < \delta < \sup_{t \in J'}t - \sup_{t \in J}t$ which will be chosen later,
    \begin{align*}
        &\left\vert \int_{X_T}\chi(t, x)\left(e^{\partial_tu_j+F(t, x, u_j)-c\psi_j} - e^{\partial_tu+F(t, x, u)-c\psi}\right)dt \wedge (\theta+dd^c\psi_j)^n \right\vert \\
        &\leq C_1 \int_{X_T \setminus (E \cap G)}\chi(t, x) \left\vert\partial_tu_j - \partial_tu \right\vert dt \wedge (\theta+dd^c\psi_j)^n \\
        & \quad + C_1 \int_{X_T \setminus (E \cap G)}\chi(t, x) \lvert F(t, x, u_j)-c\psi_j - F(t, x, u)+c\psi\rvert dt \wedge (\theta + dd^c\psi_j)^n \\
        &\leq C_1\int_{(J \times K) \setminus E}\chi(t, x) \left\vert \partial_tu_j(t, x) - \frac{u_j(t+\delta, x) - u_j(t, x)}{\delta}\right\vert dt \wedge (\theta+dd^c\psi_j)^n \\
        & \quad+C_1\int_{(J \times K) \setminus G}\chi(t, x) \left\vert\partial_tu(t, x) - \frac{u(t+\delta, x) - u(t, x)}{\delta} \right\vert dt \wedge (\theta+dd^c\psi_j)^n \\
        & \quad + C_1\int_{(J \times K)}\chi(t, x)\left\vert \frac{u_j(t+\delta, x) - u_j(t, x)}{\delta} - \frac{u(t+\delta, x) - u(t, x)}{\delta}\right\vert dt \wedge (\theta+dd^c\psi_j)^n \\
        & \quad + C_1\int_{X_T }\chi(t, x) \lvert F(t, x, u_j) - c\psi_j - F(t, x, u) + c\psi\rvert dt \wedge (\theta + dd^c\psi_j)^n \\
        & =: I_1+I_2+I_3+I_4,
    \end{align*}
    where $C_1 = \sup_{t \in J}e^{\frac{\kappa_0}{t}+M_F+\lvert c\rvert M_{\psi}} > 0$. \\
    \indent For each $j$, let $v_j(t, x) := u_j(t, x) + C_0\log t$. It follows from the concavity of the map $t \mapsto v_j(t, x)$ in $J'$ for all $x \in X$ and the well-definedness of $\partial_tu_j(t, x)$ on $(J \times K) \setminus E$ for all $j$ that 
    \begin{align}\label{inequality 1}
        \partial_tv_j(t, x) \geq \frac{v_j(t+\delta, x) - v_j(t, x)}{\delta}
    \end{align}
    for all $(t, x) \in (J \times K) \setminus E$ and $j$. Recall that $E$ satisfies (\ref{E}). Similarly, there exists a Borel subset $\widetilde{E} \subset X_T$ such that
    \begin{align*}
        \int_{\widetilde{E}\cap ((0, T) \times K)}dt \wedge (\theta + dd^c\psi_j)^n = \int_{\widetilde{E}\cap ((0, T) \times K)} dt \wedge (\theta+dd^c\psi)^n = 0
    \end{align*}
    for all $j$, $\partial_tu_j(t+\delta, x)$ exists for all $(t, x) \in (J \times K) \setminus \widetilde{E}$, and
    \begin{align}\label{inequality 2}
        \partial_tv_j(t+\delta, x) \leq \frac{v_j(t+\delta, x) - v_j(t, x)}{\delta} 
    \end{align}
    for all $(t, x) \in (J \times K) \setminus \widetilde{E}$.\\
    \indent We first derive an estimate of $I_1$. 
    \begin{align*}
        I_1 &= C_1\int_{(J \times K)\setminus E}\chi(t, x) \left\vert \partial_tv_j(t, x) - \frac{C_0}{t}-\frac{u_j(t+\delta, x) - u_j(t, x)}{\delta}\right\vert dt \wedge (\theta + dd^c\psi_j)^n \\
        &\leq C_1 \int_{(J \times K) \setminus E}\chi(t, x) \left(\partial_tv_j(t, x) - \frac{v_j(t+\delta, x) - v_j(t, x)}{\delta}\right)dt \wedge (\theta + dd^c\psi_j)^n \\
        &\quad + C_1C_0\int_{(J \times K)\setminus E}\chi(t, x) \left\vert \frac{1}{t}-\frac{\log(t+\delta) - \log t}{\delta}\right\vert dt \wedge (\theta + dd^c\psi_j)^n \\
        &\leq C_1\int_{(J \times K) \setminus (E \cap \widetilde{E})}\chi(t, x) \left(\partial_tv_j(t, x) - \partial_tv_j(t+\delta, x)\right)dt \wedge (\theta + dd^c\psi_j)^n \\
        &\quad + C_1C_0 \int_{J \times K}\chi(t, x) \left\vert \frac{1}{t}-\frac{\log(t+\delta) - \log t}{\delta}\right\vert dt \wedge (\theta + dd^c\psi_j)^n. 
    \end{align*}
    Here, the second inequality follows from (\ref{inequality 1}), and the third inequality follows from (\ref{inequality 2}). By Lemma \ref{CLN}, there exists a constant $C_2 > 0$ such that
    \begin{align}\label{CLN1}
        \int_{J \times K}dt \wedge (\theta + dd^c\psi_j)^n \leq C_2.
    \end{align}
    Now we choose
    \begin{align}\label{delta}
        \delta = \min\left\{\frac{\varepsilon \inf_{t \in J}t^2}{8C_0C_1C_2\lVert \chi\rVert_{L^{\infty}}}, \frac{\varepsilon\inf_{t \in J}t}{8C_1C_2(\kappa_0+C_0)\left\Vert \partial_t\chi\right\Vert_{L^{\infty}}}, \sup_{t \in J'}t - \sup_{t \in J}t\right\}.
    \end{align}
    By (\ref{CLN1}) and (\ref{delta}), we have
    \begin{align*}
        C_0&C_1\int_{J \times K}\chi(t, x)\left\vert \frac{1}{t}-\frac{\log(t+\delta) - \log t}{\delta}\right\vert dt \wedge (\theta + dd^c\psi_j)^n \\
        &\leq C_0C_1\int_{J \times K}\chi(t, x)\left\vert\frac{1}{t} - \frac{1}{t+\delta} \right\vert dt \wedge (\theta + dd^c\psi_j)^n \\
        &\leq C_0C_1\delta \int_{J \times K}\frac{\chi(t, x)}{t^2}dt \wedge (\theta + dd^c\psi_j)^n\\
        &\leq \frac{\varepsilon}{8}.
    \end{align*}
    Using (\ref{first derivative}), (\ref{inequality 1}), (\ref{inequality 2}), (\ref{CLN1}), (\ref{delta}), and Lemma \ref{integration by parts}, we get
    \begin{align*}
        &C_1\int_{(J \times K)\setminus (E\cap \widetilde{E})}\chi(t, x)\left(\partial_tv_j(t, x) - \partial_tv_j(t+\delta, x)\right)dt \wedge (\theta + dd^c\psi_j)^n \\
        &= C_1\int_{(J \times K) \setminus (E \cap \widetilde{E})}\partial_t\chi(t, x)(v_j(t+\delta, x) - v_j(t, x))dt \wedge (\theta + dd^c\psi_j)^n \\
        &\leq C_1\delta \int_{(J \times K) \setminus (E \cap \widetilde{E})}\left\vert \partial_t\chi(t, x)\right\vert \max\left(\left\vert \partial_tv_j(t, x)\right\vert, \left\vert\partial_tv_j(t+\delta, x)\right\vert\right)dt \wedge (\theta+dd^c\psi_j)^n \\
        &\leq C_1\delta\int_{(J \times K) \setminus (E \cap \widetilde{E})}\left\vert \partial_t\chi(t, x)\right\vert \left(\frac{\kappa_0+C_0}{t}\right) dt \wedge (\theta+dd^c\psi_j)^n\\ 
        &\leq \frac{\varepsilon}{8}.
    \end{align*}
    Indeed, the first equality holds by Lemma \ref{integration by parts}. The second inequality holds by (\ref{inequality 1}) and (\ref{inequality 2}), and the third inequality holds by (\ref{first derivative}). The last inequality holds by (\ref{CLN1}) and (\ref{delta}). Hence for all $j$,
    \begin{align}\label{I1}
        I_1 \leq \frac{\varepsilon}{4}.
    \end{align}
    Secondly, we estimate $I_2$ by the same argument for the estimate of $I_1$. For all $j$, we have
    \begin{align} \label{I2}
        I_2  \leq \frac{\varepsilon}{4}.
    \end{align}
    Next, we estimate $I_3$ as follows.
    \begin{equation}\label{I3-1}
        \begin{aligned}
            I_3 &\leq \frac{C_1}{\delta}\int_{J \times K}\chi(t, x) \lvert u_j(t+\delta, x) - u(t+\delta, x) \rvert dt \wedge (\theta + dd^c\psi_j)^n \\
        &\quad +\frac{C_1}{\delta}\int_{J \times K}\chi(t, x) \lvert u_j(t, x) - u(t, x)\rvert dt \wedge (\theta + dd^c\psi_j)^n.
        \end{aligned}
    \end{equation}
    Since $u_j(t, \cdot) \rightarrow u(t, \cdot)$ in $cap_{\theta}$ as $j \rightarrow \infty$ for a.e. $t \in (0, T)$, there exists $j_1 > 0$ such that for all $j > j_1$, 
    \begin{equation}\label{I3-2}
        \begin{aligned}
            &\frac{C_1}{\delta}\int_{J \times K}\chi(t, x) \lvert u_j(t+\delta, x)- u(t+\delta, x)\rvert dt \wedge (\theta + dd^c\psi_j)^n \\
        &\quad + \frac{C_1}{\delta}\int_{J \times K}\chi(t, x) \lvert u_j(t, x) - u(t, x)\rvert dt \wedge (\theta + dd^c\psi_j)^n \leq \frac{\varepsilon}{4}.
        \end{aligned}
    \end{equation}
    Therefore by (\ref{I3-1}) and (\ref{I3-2}), for all $j > j_1$, 
    \begin{align}\label{I3}
        I_3 \leq \frac{\varepsilon}{4}.
    \end{align}
    By the convergence of $u_j(t, \cdot) \rightarrow u(t, \cdot)$ in $cap_{\theta}$ for a.e. $t \in (0, T)$, $\psi_j \rightarrow \psi$ in $cap_{\theta}$ and continuity of $F(t, x, r)$, there exists $j_2 > 0$ such that for all $j > j_2$,
    \begin{align}\label{I4}
        I_4 \leq \frac{\varepsilon}{4}.
    \end{align}
    Combining the estimates of $I_1$, $I_2$, $I_3$, and $I_4$ in (\ref{I1}), (\ref{I2}), (\ref{I3}), and (\ref{I4}), we get
    \begin{align*}
        \left\vert \int_{X_T}\chi(t, x) \left(e^{\partial_tu_j+F(t, x, u_j)-c\psi_j} - e^{\partial_tu+F(t, x, u)-c\psi}\right)dt \wedge (\theta+dd^c\psi_j)^n\right\vert \leq \varepsilon
    \end{align*}
    for all $j > \max\{j_1, j_2\}$. \\
    \indent Next, we show that
    \begin{align}\label{result 2}
        e^{\partial_tu+F(t, x, u)-c\psi}dt \wedge (\theta + dd^c\psi_j)^n \rightarrow e^{\partial_tu+F(t, x, u)-c\psi}dt \wedge (\theta+dd^c\psi)^n \text{ weakly }
    \end{align}
    as $j \rightarrow \infty$. Indeed, we use the similar arguments in the proof of (\ref{result 1}). For $\delta$ chosen in (\ref{delta}), we have
    \begin{align*}
        &\left\vert \int_{X_T}\chi(t, x) e^{\partial_tu+F(t, x, u)-c\psi}dt \wedge \{(\theta + dd^c\psi_j)^n - (\theta + dd^c\psi)^n\}\right\vert \\
        &\quad \leq \left\vert \int_{J \times K}\chi(t, x) \left(e^{\partial_tu+F(t, x, u)-c\psi} - e^{\frac{u(t+\delta, x) - u(t, x)}{\delta}+F(t, x, u)-c\psi}\right)dt \wedge (\theta + dd^c\psi_j)^n \right\vert \\
        &\quad \quad + \left\vert \int_{J \times K}\chi(t, x)\left(e^{\partial_tu+F(t, x, u)-c\psi} - e^{\frac{u(t+\delta, x) - u(t, x)}{\delta}+F(t, x, u)-c\psi}\right)dt \wedge (\theta + dd^c\psi)^n\right\vert \\
        &\quad \quad + \left\vert \int_{J \times K}\chi(t, x) e^{\frac{u(t+\delta, x) - u(t, x)}{\delta}+F(t, x, u)-c\psi}dt \wedge \{(\theta + dd^c\psi_j)^n - (\theta + dd^c\psi)^n\}\right\vert \\
        &\quad =: I_5+I_6+I_7.
    \end{align*}
    We start with the estimate for $I_5$. 
    \begin{align*}
        I_5 &= \left\vert\int_{J \times K}\chi(t, x) e^{F(t, x, u)-c\psi}\left(e^{\partial_tu} - e^{\frac{u(t+\delta, x) - u(t, x)}{\delta}}\right)dt \wedge (\theta+dd^c\psi_j)^n\right\vert \\
        &\leq e^{M_F+\lvert c\rvert M_\psi}\int_{J \times K}\chi(t, x) \left\vert e^{\partial_tu}-e^{\frac{u(t+\delta, x) - u(t, x)}{\delta}}\right\vert dt \wedge (\theta + dd^c\psi_j)^n \\
        &\leq C_3\int_{(J \times K) \setminus G}\chi(t, x) \left\vert \partial_tu - \frac{u(t+\delta, x) - u(t, x)}{\delta}\right\vert dt \wedge (\theta + dd^c\psi_j)^n,
    \end{align*}
    where $C_3 = \sup_{t \in J}e^{\frac{\kappa_0}{t} +\frac{2\lVert u\rVert_{L^{\infty}(X_T)}}{\delta}+M_F+\lvert c\rvert M_{\psi}} > 0$. By repeating the proof for the estimate of $I_1$, we obtain that for all $j$,
    \begin{align}\label{I5}
        I_5 \leq \frac{\varepsilon}{4}.
    \end{align}
    Similarly, for all $j$,
    \begin{align}\label{I6}
        I_6 \leq \frac{\varepsilon}{4}.
    \end{align}
    Finally, we estimate $I_7$ as follows.
    \begin{align*}
        I_7 &= \left\vert \int_{J \times K}\chi(t, x) e^{\frac{u(t+\delta, x) - u(t, x)}{\delta}+F(t, x, u)-c\psi}dt \wedge \{(\theta+dd^c\psi_j)^n - (\theta+dd^c\psi)^n\}\right\vert \\
        &\leq \int_J dt \left\vert \int_K \chi(t, x)e^{\frac{u(t+\delta, x) - u(t, x)}{\delta}+F(t, x, u)-c\psi}\{(\theta+dd^c\psi_j)^n - (\theta+dd^c\psi)^n\}\right\vert.
    \end{align*}
    Let $\Phi_j(t) := \left\vert \int_K \chi(t, x)e^{\frac{u(t+\delta, x) - u(t, x)}{\delta}+F(t, x, u)-c\psi}\{(\theta + dd^c\psi_j)^n - (\theta + dd^c\psi)^n\}\right\vert$. By the definition, $\Phi_j(t) \geq 0$ for all $t \in J$ and
    \begin{equation}\label{boundedness}
        \begin{aligned}
            \Phi_j(t) &\leq \int_K\chi(t, x)e^{\frac{u(t+\delta, x) - u(t, x)}{\delta}+F(t, x, u)-c\psi}(\theta+dd^c\psi_j)^n \\
        &\quad+\int_K \chi(t, x)e^{\frac{u(t+\delta, x)-u(t, x)}{\delta}+F(t, x, u)-c\psi}(\theta+dd^c\psi)^n \\
        &\leq e^{\frac{2\lVert u\rVert_{L^{\infty}(X_T)}}{\delta}+M_F+\lvert c\rvert M_{\psi}}\lVert \chi\rVert_{L^{\infty}(X_T)}\left(\int_K (\theta+dd^c\psi_j)^n + \int_K(\theta+dd^c\psi)^n\right) \\
        &\leq 2C_4e^{\frac{2\lVert u\rVert_{L^{\infty}(X_T)}}{\delta}+M_F+\lvert c\rvert M_{\psi}}\lVert \chi\rVert_{L^{\infty}(X_T)}
        \end{aligned}
    \end{equation}
    for some uniform constant $C_4$ by the elliptic Chern-Levine-Nirenberg inequality. For each $t \in (0, T)$, the map $x \mapsto \frac{u(t+\delta, x)-u(t, x)}{\delta}+F(t, x, u)+c\psi$ is quasi-continuous in $X$ with respect to $cap_{\theta}$ since $u(t+\delta, \cdot), u(t, \cdot) \in PSH(X, B\theta)$, $\psi \in PSH(X, \theta)$ and $F(t, x, r)$ is continuous. Hence for every $t \in (0, T)$ we have
    \begin{align}\label{vanishing}
        \Phi_j(t) \rightarrow 0
    \end{align}
    as $j \rightarrow \infty$. It follows from (\ref{boundedness}) and (\ref{vanishing}) that there exists $j_3 > 0$ such that for all $j > j_3$,
    \begin{align}\label{I7}
        I_7 \leq \int_J \Phi_j(t)dt \leq \frac{\varepsilon}{2}.
    \end{align}
    Combining the estimates for $I_5$, $I_6$, and $I_7$ in (\ref{I5}), (\ref{I6}), and (\ref{I7}), we have
    \begin{align*}
        \left\vert\int_{X_T}\chi(t, x) e^{\partial_tu+F(t, x, u)-c\psi}dt \wedge \{(\theta + dd^c\psi_j)^n - (\theta+dd^c\psi)^n\}\right\vert \leq \varepsilon
    \end{align*}
    for all $j > j_3$, which implies (\ref{result 2}). \\
    \indent Finally, it follows from (\ref{result 1}) and (\ref{result 2}) that
    \begin{align*}
        &\int_{X_T}\chi(t, x) e^{\partial_tu_j+F(t, x, u)-c\psi_j}dt \wedge (\theta+dd^c\psi_j)^n \\&\rightarrow \int_{X_T}\chi(t, x) e^{\partial_tu+F(t, x, u)-c\psi}dt \wedge (\theta+dd^c\psi)^n
    \end{align*}
    as $j \rightarrow \infty$. 
\end{proof}
\section{Existence Result}
\subsection{Bounded solution}
We first give a proof of Theorem \ref{intro_main result 1}.(i).
\begin{theorem}\label{main lemma}
    Let $u_0 \in PSH(X, \omega_0) \cap L^{\infty}(X)$ and $c \in \mathbb{R}$. Assume that there exists $\psi\in PSH(X, \theta) \cap L^{\infty}(X)$ satisfying
    \begin{align*}
        (\theta+dd^c\psi)^n = e^{c\psi}d\mu \text{ in $X$}.
    \end{align*}
    Then, there exists $u \in \mathcal{P}(X_T, \omega_t) \cap L^{\infty}(X_T)$ satisfying
    \begin{align}\label{cauchy problem}
        \begin{cases}
            &dt \wedge (\omega_t+dd^cu)^n = e^{\partial_tu+F(t, x, u)}dt \wedge d\mu \text{ in $X_T$}, \\
            &\lim_{t \rightarrow 0+}u(t, \cdot) = u_0 \text{ in } L^1(X, d\mu).
        \end{cases}
    \end{align}
\end{theorem}
\begin{proof}
    It follows from \cite[Theorem 2]{BK07} that there exist two sequences $\{\lambda_j\}_{j = 1}^{\infty} \subset \mathbb{R^+}$ and $\{\psi_j\}_{j = 1}^{\infty}$ such that $\psi_j \in PSH(X, \theta + \lambda_j\omega_X)\cap C^{\infty}(X)$, $\lambda_j \downarrow 0$, and $\psi_j \downarrow \psi$ as $j \rightarrow \infty$. Let us denote by $d\nu = (\theta+dd^c\psi)^n$. We define 
    \begin{align*}
        d\nu_j := (\theta+2\lambda_j\omega_X+dd^c\psi_j)^n
    \end{align*} so that $d\nu_j \rightarrow d\nu$ weakly as $j \rightarrow \infty$. Indeed, let $\{U_k\}_{k = 1}^m$ be an open cover of $X$ and fix $\{\chi_k\}_{k = 1}^m$ be a partition of unity subordinate to $\{U_k\}_{k = 1}^m$. Let $\phi_k(x)$ and $\rho_k(x)$ be local potentials of $\theta$ and $\omega_X$ on $U_k$ for each $k$ and let $\chi(x)$ be a test function on $X$. We have
    \begin{equation}\label{potential}
        \begin{aligned}
            &\lim_{j \rightarrow \infty}\int_{X}\chi(x) (\theta+2\lambda_j\omega_X+dd^c\psi_j)^n \\
            &\quad= \lim_{j \rightarrow \infty}\sum_{k = 1}^m \int_{U_k}\chi(x) (dd^c\phi_k+2\lambda_jdd^c\rho_k+dd^c\psi_j)^n \\
        &\quad= \sum_{k = 1}^m\int_{U_k}\chi(x) (dd^c\phi_k+dd^c\psi)^n \\
        &\quad= \int_{X}\chi (\theta+dd^c\psi)^n.
        \end{aligned}
    \end{equation}
    The third equality in (\ref{potential}) holds by using a partition of unity subordinate to $\{U_k\}_{k = 1}^m$. Since $d\mu_j \geq \lambda_j^n \omega_X^n > 0$, it follows from \cite[Corollary 4.5]{GLZ20} that for each $j$, there exists $u_j \in \mathcal{P}(X_T, \omega_t+2\lambda_j\omega_X) \cap C(X_T)$ such that
    \begin{align}\label{cauchy problem2}
        \begin{cases}
            &dt \wedge (\omega_t+2\lambda_j\omega_X+dd^cu_j)^n = e^{\partial_tu_j+\Tilde{F}_j(t, x, u_j)}dt \wedge d\nu_j \text{ in $X_T$},\\
            &\lim_{t \rightarrow 0+}u_j(t, \cdot) = u_0 \text{ in } L^1(X, dV),
        \end{cases}
    \end{align}
    where $\Tilde{F}_j(t, x, r) = F(t, x, r) - c\psi_j(x)$.
    Since the sequence $\{\psi_j\}_{j = 1}^{\infty}$ is uniformly bounded in $X$, it follows from Lemma \ref{uniform estimate} that $\{u_j\}_{j = 1}^{\infty}$ is uniformly bounded in $X_T$. We pick a subsequence of $\{u_j\}_{j = 1}^{\infty}$ such that $u_j \rightarrow u$ in $L^1(X, dV)$ for some $u \in \mathcal{P}(X_T, \omega_t) \cap L^{\infty}(X_T)$. We show that $u$ is the desired solution. \\
    \indent First, we show that 
    \begin{align}\label{4.4}
        dt \wedge (\omega_t+2\lambda_j\omega_X+dd^cu_j)^n \rightarrow dt \wedge (\omega_t+dd^cu)^n \text{ weakly }
    \end{align}
    as $j \rightarrow \infty$. For a.e. $t \in (0, T)$, 
    \begin{align*}
        (\omega_t+2\lambda_j\omega_X+dd^cu_j(t, \cdot))^n \leq e^{C_0}(\theta+2\lambda_j \omega_X + dd^c\psi_j)^n \leq e^{C_0}(\omega_t+2\lambda_j \omega_X + dd^c\psi_j)^n.
    \end{align*}
    for some constant $C_0 > 0$. We first show that for a.e. $t \in (0, T)$ and for all $j_0 > 0$, 
    \begin{align} \label{4.5}
        u_j(t, \cdot) \rightarrow u(t, \cdot) \text{ in } cap_{\omega_t+2\lambda_{j_0}\omega_X}
    \end{align} as $j \rightarrow \infty$, by following the proof of \cite[Theorem 3.8]{DP12}. Let us fix $j_0$ and denote $\lambda := \lambda_{j_0}$. For a.e. $t \in (0, T)$ and for all $j > j_0$, 
    \begin{align}\label{inequality}
        (\omega_t+2\lambda\omega_X+dd^cu_j(t, \cdot))^n \leq e^{C_0}(\omega_t+2\lambda\omega_X+dd^c\psi_j)^n
    \end{align}
    and $\psi_j \rightarrow \psi$ in $cap_{\omega_t+2\lambda
    \omega_X}$ as $j \rightarrow \infty$. Fix such a $t$. It suffices to show that for all $\delta > 0$,
    \begin{align*}
        \int_{\{u_j(t, \cdot) < u(t, \cdot) - \delta\}}(\omega_t+2\lambda\omega_X+dd^cu_j(t, \cdot))^n \rightarrow 0
    \end{align*}
    as $j \rightarrow \infty$. Fix $\delta > 0$ and $\varepsilon > 0$. It follows from (\ref{inequality}) that 
    \begin{align*}
        &\int_{\{u_j(t, \cdot) < u(t, \cdot) - \delta\}}(\omega_t+2\lambda\omega_X+dd^cu_j(t, \cdot))^n \\&\quad\leq e^{C_0}\int_{\{u_j(t, \cdot) < u(t, \cdot) - \delta\}}(\omega_t+2\lambda \omega_X+dd^c\psi_j)^n \\
        &\quad\leq \frac{e^{C_0}}{\delta}\int_{X}\lvert u(t, \cdot) - u_j(t, \cdot)\rvert (\omega_t+2\lambda\omega_X+dd^c\psi_j)^n \\
        &\quad\leq \frac{e^{C_0}}{\delta}\int_{X}\lvert u(t, \cdot) + \varepsilon - u_j(t, \cdot)\rvert (\omega_t+2\lambda\omega_X+dd^c\psi_j)^n + \frac{e^{C_0}\varepsilon}{\delta}\int_{X}(\omega_t+2\lambda\omega_X)^n.
    \end{align*}
    Here the first term is bounded above by
    \begin{align*}
        \frac{e^{C_0}}{\delta}(I_1+I_2+I_3)
    \end{align*}
    where
    \begin{align*}
        I_1 &:= \int_{X}(u(t, \cdot)-u_j(t, \cdot))(\omega_t+2\lambda\omega_X+dd^c\psi_j)^n, \\
        I_2 &:=  (\lVert u_j\rVert_{L^{\infty}(X_T)}+\lVert u\rVert_{L^{\infty}(X_T)} + \varepsilon)\int_{\{u_j(t, \cdot) > u(t, \cdot)+\varepsilon\}}(\omega_t+2\lambda\omega_X+dd^c\psi_j)^n, \\
        I_3 &:= \varepsilon \int_{X}(\omega_t+2\lambda\omega_X)^n.
    \end{align*}
    It follows from the uniform boundedness of $\{u_j\}_{j = 1}^{\infty}$ and Hartogs lemma that $I_2 \rightarrow 0$ as $j \rightarrow \infty$. Hence it suffices to show that $I_1 \rightarrow 0$ as $j \rightarrow \infty$. Note that
    \begin{align*}
        &\int_{X}(u(t, \cdot)-u_j(t, \cdot))(\omega_t+2\lambda\omega_X+dd^c\psi_j)^n\\ &\quad= \int_{X}(u(t, \cdot)-u_j(t, \cdot))\{(\omega_t+2\lambda\omega_X+dd^c\psi_j)^n - (\omega_t+2\lambda\omega_X+dd^c\psi)^n\}\\ &\quad\quad+ \int_X (u(t, \cdot)-u_j(t, \cdot))(\omega_t+2\lambda\omega_X+dd^c\psi)^n \\
        &\quad = \int_X (\psi_j - \psi)\{(\omega_t+2\lambda\omega_X+dd^cu_j(t, \cdot)) - (\omega_t+2\lambda\omega_X+dd^cu(t, \cdot))\} \wedge T_j \\
        &\quad\quad + \int_X(u(t, \cdot)-u_j(t, \cdot))(\omega_t+2\lambda\omega_X+dd^c\psi)^n
    \end{align*}
    where
    \begin{align*}
        T_j := \sum_{k = 0}^{n-1}(\omega_t+2\lambda\omega_X+dd^c\psi_j)^k \wedge (\omega_t+2\lambda\omega_X+dd^c\psi)^{n-k-1}.
    \end{align*}
    It follows from \cite[Lemma 3.1]{Ce12} that the second term converges to $0$ as $j \rightarrow \infty$. The first term is bounded above by
    \begin{equation}\label{upper bound}
        \begin{aligned}
            &\int_X\lvert \psi_j-\psi\rvert \{(\omega_t+2\lambda\omega_X+dd^cu_j(t, \cdot))+(\omega_t+2\lambda\omega_X+dd^cu(t, \cdot))\}\wedge T_j \\
        &\quad = 2\int_X\lvert \psi_j-\psi\rvert (\omega_t+2\lambda\omega_X+dd^c(u_j(t, \cdot)+u(t, \cdot))) \wedge T_j,    
        \end{aligned}
    \end{equation}
    and the measure $(\omega_t+2\lambda\omega_X+dd^c(u_j(t, \cdot)+u(t, \cdot))) \wedge T_j$ is uniformly absolutely continuous with respect to $cap_{\omega_t+2\lambda\omega_X}$. Since $\psi_j \downarrow \psi$, (\ref{upper bound}) converges to $0$ as $j \rightarrow \infty$.\\ \indent We now show (\ref{4.4}) by using the convergence (\ref{4.5}). Let us fix $t \in (0, T)$ such that (\ref{4.5}) holds. For a test function $\chi(x)$ on $X$,
    \begin{align*}
        &\left\vert\int_{X}\chi(x)(\omega_t+2\lambda_j\omega_X+dd^cu_j(t, \cdot))^n - \int_{X}\chi(x)(\omega_t+dd^cu(t, \cdot))^n\right\vert \\&\quad\leq \left\vert \int_{X}\chi(x)\left\{(\omega_t+2\lambda_j\omega_X+dd^cu_j)^n - (\omega_t+2\lambda_{j_0}\omega_X+dd^cu_j)^n\right\} \right\vert \\
        &\quad\quad+\left\vert \int_{X}\chi(x) \left\{(\omega_t+2\lambda_{j_0}\omega_X+dd^cu_j)^n-(\omega_t+2\lambda_{j_0}\omega_X+dd^cu)^n\right\} \right\vert\\
        &\quad\quad + \left\vert\int_{X}\chi(x)\left\{(\omega_t+2\lambda_{j_0}\omega_X+dd^cu)^n - (\omega_t+dd^cu)^n\right\} \right\vert \\
        &\quad=:S_1+S_2+S_3.
    \end{align*}
    We first prove the convergence of $S_1$. Let us assume that $j \geq j_0$. Recall that $\{\chi_k\}_{k = 1}^m$ is a partition of unity subordinate to $\{U_k\}_{k = 1}^m$. \begin{equation}\label{S_1}
        \begin{aligned}
            S_1 &= \left\vert\sum_{k = 1}^m\int_{U_k}\chi_k(x)\chi(x)\{(dd^c(\phi_k+2\lambda_j\rho_k+u_j))^n - (dd^c(\phi_k+2\lambda_{j_0}\rho_k+u_j))^n\} \right\vert\\
        &\leq 2\sum_{k = 1}^m\left\vert\int_{U_k}\chi_k(x)\chi(x) dd^c(2(\lambda_{j_0}-\lambda_{j})\rho_k) \wedge (dd^c(\phi_k+2\lambda_{j_0}\rho_k+u_{j}))^{n-1} \right\vert \\
        &\leq 4\lvert \lambda_{j_0}-\lambda_j\rvert\sum_{k =1}^m\lVert \chi_k\chi\rVert_{L^{\infty}(U_k)}\int_{U_k}dd^c\rho_k \wedge (dd^c(\phi_k+2\lambda_{j_0}\rho_k+ u_j))^{n-1}.
        \end{aligned}
    \end{equation}
    Since $\rho_k(x)$, $\phi_k(x)$ and $u_j(t, x)$ are uniformly bounded, it follows from the elliptic Chern-Levine-Nirenberg inequality that $S_1 \rightarrow 0$ as $j, j_0 \rightarrow \infty$. Next, the convergence of $S_3$ follows from the similar argument with the proof for $S_1$. Finally, it follows from (\ref{4.5}) that $S_2 \rightarrow 0$ as $j \rightarrow \infty$.  \\
    \indent Next, we show that 
    \begin{align}\label{righthandside}
        e^{\partial_tu_j+\Tilde{F}_j(t, x, u_j)}dt \wedge d\nu_j \rightarrow e^{\partial_tu+F(t, x, u)}dt \wedge d\mu \text{ weakly}
    \end{align}
    as $j \rightarrow \infty$, or equivalently, 
    \begin{align*}
        e^{\partial_tu_j+F(t, x, u)-c\psi_j}dt \wedge (\theta+2\lambda_j\omega_X+dd^c\psi_j)^n \rightarrow e^{\partial_tu+F(t, x, u)-c\psi}dt \wedge (\theta+dd^c\psi)^n \text{ weakly }
    \end{align*}
    as $j \rightarrow \infty$. We proceed as the proof for the convergence (\ref{4.4}). Let us fix $j_0 > 0$. We have
    \begin{align*}
        e^{\partial_tu_j+\Tilde{F}_j(t, x, u_j)}dt \wedge d\nu_j - e^{\partial_tu+F(t, x, u)}dt \wedge d\mu =: A_1+A_2+A_3
    \end{align*}
    where
    \begin{align*}
        A_1 &:= e^{\partial_tu_j+\Tilde{F}_j(t, x, u_j)}dt \wedge\left\{(\theta+2\lambda_j\omega_X+dd^c\psi_j)^n - (\theta+2\lambda_{j_0}\omega_X+dd^c\psi_j)^n\right\}, \\
        A_2 &:= e^{\partial_tu_j+\Tilde{F}_j(t, x, u_j)}dt \wedge (\theta+2\lambda_{j_0}\omega_X+dd^c\psi_{j})^n \\
        &\quad- e^{\partial_tu+F(t, x, u)-c\psi}dt \wedge (\theta+2\lambda_{j_0}\omega_X+dd^c\psi)^n, \\
        A_3 &:= e^{\partial_tu+F(t, x, u)-c\psi}dt \wedge \left\{(\theta+2\lambda_{j_0}\omega_X+dd^c\psi)^n- (\theta+dd^c\psi)^n\right\}.
    \end{align*}
    The weak convergence $A_1 \rightarrow 0$ and $A_3 \rightarrow 0$ as $j \rightarrow \infty$ follows from the similar argument with the proof of the convergence of $S_1$ in (\ref{S_1}). Hence it suffices to prove the weak convergence of $A_2 \rightarrow 0$ as $j \rightarrow \infty$.
    By Lemma \ref{lipschitzness}, there exists a constant $\kappa > 0$
    \begin{align}\label{estimate on first derivative}
        \left\vert\partial_tu_j(t, x)\right\vert \leq \frac{\kappa}{t} \text{ for all }j
    \end{align}
    for all $(t, x) \in X_T$ where $\partial_tu_j(t, x)$ is well-defined. Moreover, by Lemma \ref{concavity1}, for all intervals $J \Subset (0, T)$ there exists a constant $C = C(J)> 0$ such that 
    \begin{align*}
        t \mapsto u_j(t, x) \text{ is concave in } J
    \end{align*}
    for all $x \in X$ and $j$.
    It follows from Lemma \ref{convergence} that 
    \begin{align*}
        &e^{\partial_tu_j+F(t, x, u_j)-c\psi_j}dt \wedge (\theta+2\lambda_{j_0}\omega_X+dd^c\psi_j)^n \\
        &\rightarrow e^{\partial_tu+F(t, x, u)-c\psi}dt \wedge (\theta+2\lambda_{j_0}\omega_X+dd^c\psi)^n
    \end{align*} 
    weakly as $j \rightarrow \infty$.
    \\
    \indent Finally, we show that
    \begin{align*}
        \lim_{t \rightarrow 0+}u(t, \cdot) = u_0 \text{ in } L^1(X, d\mu).
    \end{align*}
    It follows from \cite[Proposition 2.2]{GLZ20} that
    \begin{align*}
        u_j(t, x) \geq (1-t)e^{-At}u_0+t\psi_j+n(t\log t - t) - C_1\frac{e^{\lambda_{F} t}-1}{\lambda_{F}}
    \end{align*}
    for all $0 \leq t \leq 1$ and $j$, where $C_1$ is a constant depending only on $\sup_{X_T}\Tilde{F}_j(t, x, 0)$, $\lambda_{\Tilde{F}_j} = \lambda_F$, $A$, $\lVert u_0\rVert_{L^{\infty}(X)}$ and $\lVert \psi_j \rVert_{L^{\infty}(X)}$. Hence by taking $j \rightarrow \infty$ and $t \rightarrow 0+$, we get
    \begin{align*}
        \lim_{t \rightarrow 0+}u(t, x) \geq u_0(x)
    \end{align*}
    for all $x \in X$. \\
    \indent Next, fix an open set $U \subset X$ and let $\chi(x)$ be a positive continuous test function on $X$. Let $C_2 = -\inf_{X_T \times [-\sup_j\lVert u_j\rVert_{L^{\infty}(X_T)}, \sup_j\lVert u_j\rVert_{L^{\infty}(X_T)}]}F(t, x, r)$. Let us denote by $d\mu_j = e^{-c\psi_j}(\theta+2\lambda_j\omega_X+dd^c\psi_j)^n$ so that $d\mu_j \rightarrow d\mu$ weakly as $j \rightarrow \infty$. We claim that the function $t \mapsto \int_U \chi u_j(t, \cdot)d\mu_j - C_3t$ is nonincreasing for some constant $C_3 > 0$. Since $u_j$ satisfies (\ref{cauchy problem2}), for a.e. $t \in (0, T)$,
    \begin{align*}
        \int_U \chi(x) e^{\partial_tu_j(t, \cdot)-C_2}d\mu_j \leq \int_X\omega_t^n \leq B^n\int_X \theta^n.
    \end{align*}
    If $\mu_j(U) = 0$, we are done. Otherwise, by Jensen's inequality, we have
    \begin{align*}
        \int_U \chi(x) e^{\partial_tu_j(t, \cdot)}\frac{d\mu_j}{\mu_j(U)} \geq \exp \left(\int_U \partial_tu_j(t, \cdot)\frac{\chi(x) d\mu_j}{\mu_j(U)}\right)
    \end{align*}
    Hence we get
    \begin{align*}
        \int_U \chi(x)\partial_tu_j(t, \cdot)d\mu_j \leq  \mu_j(U)\left(n\log B + C_2 + \log \int_X \theta^n\right) \leq C_3.
    \end{align*}
    for some constant $C_3 > 0$. Therefore the function $t \mapsto \int_U \chi(x) u_j(t, \cdot)d\nu_j - C_2t$ is nonincreasing, which implies that
    \begin{align*}
        \int_U \chi(x) u_j(t, \cdot)d\mu_j \leq \int_U \chi(x) u_0d\mu_j + C_2t.
    \end{align*}
    Letting $j \rightarrow \infty$, we get
    \begin{align*}
        \int_U \chi(x) u(t, \cdot) d\mu \leq \int_U \chi(x) u_0d\mu + C_2t.
    \end{align*}
    Indeed, it holds since $u_j(t, \cdot) \rightarrow u(t, \cdot)$ in $cap_{\theta+2\lambda_{j_0}\omega_X}$ and $d\mu_j \rightarrow d\mu$ weakly as $j \rightarrow \infty$ for a.e. $t \in (0, T)$ and for all $j_0 > 0$. \\
    \indent
    Since $u(t, x)$ is bounded, there exist a sequence of positive numbers $\{t_k\}_{k = 1}^{\infty}$ and a function $w \in PSH(X, \theta)$ such that $t_k \rightarrow 0$ as $k \rightarrow \infty$ and for a.e. $x \in X$ with respect to $\theta^n$, $u(t_k, x) \rightarrow w(x)$ as $k \rightarrow \infty$. We claim that for $\mu$-a.e. $x \in U$, $u(t_k, x) \rightarrow w(x)$ as $k \rightarrow \infty$. Indeed, we have $\sup_k \lVert u(t_k, \cdot)\rVert_{L^2(U, d\mu)} < \infty$. Hence there exists a function $W \in L^2(U, d\mu)$ such that by passing to a subsequence, $u(t_k, \cdot) \rightarrow W$ weakly as $k \rightarrow \infty$. It follows from the proof of \cite[Lemma 2.1]{KN23} that again by passing to a subsequence,
    \begin{align*}
        \lim_{k \rightarrow \infty}\int_U u(t_k, \cdot)d\mu = \int_U W d\mu = \int_U wd\mu.
    \end{align*}
    Let us denote this subsequence by $\{t_k\}_{k = 1}^{\infty}$. Since $d\mu = e^{-c\psi}(\theta+dd^c\psi)^n$ and $\psi$ is bounded, it follows from the proof of \cite[Corollary 2.2]{KN23} that
    \begin{align*}
        \lim_{k \rightarrow \infty}\int_U \lvert u(t_k, \cdot) - w\rvert d\mu = 0.
    \end{align*}
    Hence we have $w(x) \leq u_0(x)$ for $\mu$-a.e. $x \in U$. Since $U$ is chosen arbitrarily, $w \leq u_0$ $\mu$-a.e. in $X$. Therefore $\lim_{t \rightarrow 0+}u(t, \cdot) = u_0$ in $L^1(X, d\mu)$.
\end{proof}
\begin{remark}\label{admissible pair}
    While proving the boundary condition in Lemma \ref{main lemma}, we actually obtained that $\lim_{t \rightarrow 0+}u(t, x) \geq u_0(x)$ for all $x \in X$ and $\lim_{t \rightarrow 0+}u(t, x) \leq u_0(x)$ for $\mu$-a.e. in $X$. Hence, we can apply the results in \cite[Section 5.2]{Ka25}. More precisely, let us say that the pair $(\mu, u_0)$ of a measure $\mu$ and a function $u_0$ is admissible if for any $v \in PSH(X, \theta) \cap L^{\infty}(X)$, $v \leq u_0$ $d\mu$-a.e. in $X$ implies that $v \leq u_0$ in $X$. If $(\mu, u_0)$ is admissible, then the proof of Lemma \ref{main lemma} implies that the solution $u$ that we obtained actually satisfies $\lim_{t \rightarrow 0+}u(t, x) = u_0(x)$ for all $x \in X$.
\end{remark}
    From now, we only consider the case when $c = 0$. We define a general criterion to solve the Monge-Amp\`ere equation for the condition in Theorem \ref{main lemma} above. 
\begin{definition}
    Let $\mathcal{H}_{\theta}$ be a class of finite positive Borel measures $\mu$ on $X$ such that for all Borel subsets $K \subset X$,
    \begin{align*}
        \mu(K) \leq C(cap_{\theta}(K))^{1+\alpha}
    \end{align*}
    for some $C > 0$ and $\alpha > 0$.
\end{definition}
For example, if $\mu = g\theta^n$ with $0 \leq g \in L^p(X, \theta^n)$ for some $p > 1$, then $\mu \in \mathcal{H}_{\theta}$. Also, under the assumption that $\theta$ is K\"ahler, it follows from \cite[Corollary 1.2]{DNS14} that if $\mu$ is a Monge-Amp\`ere measure of some H\"older continuous $\theta$-plurisubharmonic function, then $\mu \in \mathcal{H}_{\theta}$. With our original assumption on $\theta$ that it is closed, semi-positive and big, it follows from \cite[Theorem 2.1]{EGZ09} that if $\mu \in \mathcal{H}_{\theta}$, then there exists $\psi \in PSH(X, \theta) \cap C(X)$ satisfying
\begin{align*}
    (\theta+dd^c\psi)^n = d\mu \text{ in }X.
\end{align*}
Hence the next lemma follows from Theorem \ref{main lemma}.
\begin{lemma}\label{H theta}
    Let $u_0 \in PSH(X, \omega_0) \cap L^{\infty}(X)$. If $\mu \in \mathcal{H}_{\theta}$, then there exists $u \in \mathcal{P}(X_T, \omega_t) \cap L^{\infty}(X_T)$ satisfying
    \begin{align*}
        \begin{cases}
            & dt \wedge (\omega_t+dd^cu)^n = e^{\partial_tu+F(t, x, u)}dt \wedge d\mu \text{ in $X_T$}, \\
            & \lim_{t \rightarrow 0+}u(t, \cdot) = u_0 \text{ in } L^1(X, d\mu).
        \end{cases}
    \end{align*}
\end{lemma}
We now prove a partial result of the Theorem \ref{intro_main result 1}.(ii).
\begin{lemma}\label{main theorem}
       Let $u_0 \in PSH(X, \omega_0) \cap L^{\infty}(X)$. Assume that there exist a H\"older continuous $\omega_X$-psh function $\phi$ and a constant $C > 0$ satisfying
    \begin{align*}
        d\mu \leq C(\omega_X+dd^c\phi)^n \text{ in $X$}.
    \end{align*}
    Then, there exists $u \in \mathcal{P}(X_T, \omega_t) \cap L^{\infty}(X_T)$ satisfying
    \begin{align}\label{4.9}
        \begin{cases}
            &dt \wedge (\omega_t+dd^cu)^n = e^{\partial_tu+F(t, x, u)}dt \wedge d\mu \text{ in $X_T$}, \\
            &\lim_{t \rightarrow 0+}u(t, \cdot) = u_0 \text{ in } L^1(X, d\mu). 
        \end{cases}
    \end{align}
\end{lemma}
\begin{proof}
        Let us fix $\tau > n$. By \cite[Corollary 2.2]{KN18}, there exists $C_{\tau} > 0$ such that for every Borel subset $K \subset X$,
    \begin{align}\label{H(tau)}
        \mu(K) \leq C_{\tau}[cap_{\omega_X}(K)]^{1+\tau}.
    \end{align}
    It follows from \cite[Theorem 4.6]{Din15} (see also \cite[Corollary 12.4]{GZ17}) that $cap_{\omega_X} \leq C_1cap_{\theta}^{1/n}$ for some $C_1 > 0$. Therefore we have
    \begin{align*}
        \mu(K) \leq C_1C_{\tau}[cap_{\theta}(K)]^{\frac{1+\tau}{n}} = C_{\tau'}[cap_{\theta}(K)]^{1+\tau'}
    \end{align*}
    where $C_{\tau'} := C_1C_{\tau}$ and $\tau' := \frac{1+\tau}{n}-1 > 0$ since $\tau > n$. Since it implies that $\mu \in \mathcal{H}_{\theta}$, the result follows from Lemma \ref{H theta}.
\end{proof}
\subsection{H\"older continuous solution}
We now show that the solution obtained from Lemma \ref{main theorem} is locally H\"older continuous on $\rm{Amp}(\theta)$, which completes the proof of Theorem \ref{intro_main result 1}.(ii). We follow the proof of \cite[Theorem D]{DDGPZ14}, combining with the results in \cite{KN18}.
\begin{theorem}\label{Holder Main}
    Let $u_0 \in PSH(X, \omega_0) \cap L^{\infty}(X)$. Assume that there exist a H\"older continuous $\omega_X$-psh function $\phi$ and a constant $C > 0$ satisfying
    \begin{align*}
        d\mu \leq C(\omega_X+dd^c\phi)^n \text{ in $X$}.
    \end{align*}
    Then, there exists $u \in \mathcal{P}(X_T, \omega_t) \cap L^{\infty}(X_T)$ satisfying
    \begin{align*}
        \begin{cases}
            &dt \wedge (\omega_t+dd^cu)^n = e^{\partial_tu+F(t, x, u)}dt \wedge d\mu \text{ in $X_T$}, \\
            &\lim_{t \rightarrow 0+}u(t, \cdot) = u_0 \text{ in } L^1(X, d\mu), \\
            &u \text{ is locally H\"older continuous on } {\rm{Amp}}(\theta) \text{ for all } t \in (0, T),
        \end{cases}
    \end{align*}
    where the H\"older exponent of $u(t, \cdot)$ does not depend on $t$. Moreover, $u$ is jointly continuous on $(0, T) \times {\rm{Amp}}(\theta)$.
\end{theorem}
\begin{remark}
    If $\theta$ is K\"ahler, then $\rm{Amp}(\theta) = X$. Therefore in this case, we obtain a solution $u(t, \cdot)$ which is locally H\"older continuous on $X$ for all $t \in (0, T)$.
\end{remark}
\begin{remark}
    The modulus of continuity of the solution can be also obtained from the recent result of \cite{DDP25}.
\end{remark}
\begin{proof}[Proof of Theorem \ref{Holder Main}]
    \indent It follows from Lemma \ref{main theorem} that there exists $u \in \mathcal{P}(X_T, \omega_t) \cap L^{\infty}(X_T)$ satisfying
    \begin{align*}
        \begin{cases}
            &dt \wedge (\omega_t+dd^cu)^n = e^{\partial_tu+F(t, x, u)}dt \wedge d\mu \text{ in $X_T$}, \\
            &\lim_{t \rightarrow 0+}u(t, \cdot) = u_0 \text{ in } L^1(X, d\mu).
        \end{cases}
    \end{align*}
    Let us fix $t \in (0, T)$ satisfying
    \begin{align*}
        (\omega_t+dd^cu(t, \cdot))^n = e^{\partial_tu(t, \cdot)+F(t, \cdot ,u)}d\mu \text{ in $X$}.
    \end{align*}
    We now show that $u(t, \cdot)$ is locally H\"older continuous on $\rm{Amp}(\theta)$. We remark that when $\theta$ is K\"ahler, the result follows directly by the famous result of \cite{DDGPZ14}. Indeed, since $\theta \leq \omega_t \leq B\theta$ for all $t$, $\omega_t$ is also K\"ahler for all $t$. It implies
    \begin{align*}
        (\omega_t+dd^cu(t, \cdot))^n \leq C(\omega_t+C'dd^c\phi)^n
    \end{align*}
    for some $C, C' > 0$. Therefore, it follows from \cite[Theorem 4.3]{DDGPZ14} that $u(t, \cdot)$ is H\"older continuous on $X$. In our setting, $\theta$ and $\omega_t$ are merely semi-positive, but we can still use the same method. We provide a detailed argument for the sake of completeness. \\
    \indent We may assume that $-C_0 \leq u(t, \cdot) \leq 0$ for some constant $C_0 > 0$. Let
\begin{align*}
    u(t, \cdot) \mapsto \rho_{\delta}u(t, \cdot)
\end{align*}
denote the regularization operator defined in \cite{DDGPZ14}. For some $\delta_0 > 0$ and $K > 0$, a map $s \mapsto \rho_{s}u(t, \cdot)+Ks^2$ is increasing for $0 \leq s \leq \delta_0$. We fix such $\delta_0$ and $K$. \\
\indent Let us consider the Kiselman-Legendre transform
\begin{align*}
    U_{c, \delta}(t, \cdot) := \inf_{s \in (0, \delta]}\{\rho_s u(t, \cdot)+Ks^2-c\log(s/\delta)\},
\end{align*}
where $0 < \delta < \delta_0$ and $c > 0$ will be chosen later. We have 
\begin{align*}
    u(t, \cdot) \leq U_{c, \delta}(t, \cdot) \leq \rho_\delta u(t, \cdot) + K\delta^2.
\end{align*}
It follows from the estimate in \cite[Lemma 2.1]{DDGPZ14} that 
\begin{align*}
    \omega_t+dd^cU_{c, \delta}(t, \cdot) \geq -(Ac+K\delta^2)\omega_X
\end{align*}
for some constant $A > 0$. Since $\theta$ is a big form, there exists a $\theta$-psh function $U_0$ on $X$ such that $\theta+dd^cU_0 \geq \varepsilon_0\omega_X$ for some small constant $\varepsilon_0 > 0$. We may assume that $U_0 \leq 0$. \\
\indent A function $u_{c, \delta}$ defined by
\begin{align*}
    u_{c, \delta}(t, \cdot) := \frac{Ac+K\delta^2}{\varepsilon_0}U_0 + \left(1-\frac{Ac+K\delta^2}{\varepsilon_0}\right)U_{c, \delta}(t, \cdot)
\end{align*}
is $\omega_t$-plurisubharmonic on $X$. It follows from \cite[Theorem 4.3]{DDGPZ14} (see also \cite[Lemma 3.3]{DN14}) that there exist $C_1 > 0$ and $0 < \alpha_1 < 1$ such that
\begin{align} \label{L1mu}
    \lVert \rho_{\delta}u(t, \cdot) - u(t, \cdot)\rVert_{L^1(d\mu)} \leq C_1\delta^{\alpha_1}.
\end{align}
\indent We fix $0 < \delta < \delta_0$ and define
\begin{align*}
    \alpha = \min\left\{\frac{\tau'}{(n+1)\tau'+n}, \alpha_1\right\}.
\end{align*}
Let us choose $c > 0$ such that
\begin{align*}
    Ac+K\delta^2 = \varepsilon_0\delta^{\alpha_1\alpha}.
\end{align*}
From now, we denote by
\begin{align*}
    u_{\delta}(t, \cdot) := u_{c, \delta}(t, \cdot).
\end{align*}
Since $U_0 \leq 0 \leq u(t, \cdot) + C_0$, we have 
\begin{align*}
    u_{\delta}(t, \cdot) - u(t, \cdot) &= \delta^{\alpha_1\alpha}(U_0 - u(t, \cdot))+(1-\delta^{\alpha_1\alpha})(U_{c, \delta}(t, \cdot) - u(t, \cdot)) \\
    & \leq C_0\delta^{\alpha_1\alpha} + (1-\delta^{\alpha_1\alpha})(\rho_{\delta}u(t, \cdot) - u(t, \cdot)+K\delta^2). 
\end{align*}
Moreover, since $u(t, \cdot) \leq 0$, we get $\rho_{\delta}u(t, \cdot) \leq 0$. By shrinking $\delta_0$ if necessary, we have
\begin{align*}
    U_{c, \delta}(t, \cdot) \leq K\delta^2 \leq C_0\delta^{\alpha_1\alpha}    
\end{align*}
for all $0 < \delta < \delta_0$. Therefore we have
\begin{align*}
    u_{\delta}(t, \cdot) = \delta^{\alpha_1\alpha}U_0 + (1-\delta^{\alpha_1\alpha})U_{c, \delta}(t, \cdot) \leq C_0\delta^{\alpha_1\alpha},
\end{align*}
thus $u_{\delta}(t, \cdot) - C_0\delta^{\alpha_1\alpha} \leq 0$.
It follows from \cite[Proposition 5.3]{GZ12} that for $\gamma = \frac{\tau'}{(n+1)\tau'+n}$ and $d\nu = e^{\partial_tu(t, \cdot)+F(t, \cdot ,u)}d\mu$,
\begin{align*}
    \sup_{X}(u_{\delta}(t, \cdot) - u(t, \cdot)) &\leq B_0\lVert \max(u_{\delta}(t, \cdot) - u(t, \cdot) - C_0\delta^{\alpha_1\alpha}, 0)\rVert_{L^1(d\nu)}^{\gamma} + C_0\delta^{\alpha_1\alpha} \\
    &\leq B_0 \lVert \rho_{\delta}u(t, \cdot)- u(t, \cdot) + K\delta^2 \rVert_{L^1(d\nu)}^{\gamma}+C_0\delta^{\alpha_1\alpha}
\end{align*}
for some constant $B_0 > 0$ which depends on $\gamma$ and the uniform norm of $u(t, \cdot)$. By (\ref{L1mu}), the last estimate yields
\begin{align}\label{parameter bound}
    \sup_{X}(u_{\delta}(t, \cdot) - u(t, \cdot)) \leq C_2\delta^{\alpha_1\alpha},
\end{align}
for some $C_2 > 0$. Indeed, for $\kappa$ obtained in (\ref{estimate on first derivative}), we have
\begin{align*}
    &B_0\left(\int_{X}\lvert \rho_{\delta}u(t, \cdot)+K\delta^2-u(t, \cdot)\rvert e^{\partial_tu(t, \cdot)+F(t, \cdot ,u)}d\mu\right)^\gamma + C_0\delta^{\alpha_1\alpha} \\
    &\quad\leq B_0e^{\gamma(t^{-1}\kappa+\lVert F(t, x, u(t, x))\rVert_{L^{\infty}(X_T)})}(\lVert \rho_{\delta}u(t, \cdot)-u(t, \cdot)\rVert_{L^1(d\mu)}^\gamma+K^{\gamma}\delta^{2\gamma}) + C_0\delta^{\alpha_1\alpha} \\
    &\quad \leq C_2\delta^{\alpha_1\alpha}.
\end{align*}
where $C_2 := B_0e^{\gamma(t^{-1}\kappa+\lVert F(t, x, u(t, x))\rVert_{L^{\infty}(X_T)})}(C_1+K^{\gamma})+C_0$.
\\
\indent Next, the inequality (\ref{parameter bound}) yields a uniform lower bound on the parameter $\Tilde{s} = \Tilde{s}(z)$ which attains the infimum in the definition of $u_{\delta, t}(z)$ for a fixed $z \in \rm{Amp}(\theta)$. Indeed, by the definition of $u_{\delta, t}$,
\begin{align*}
    u_{\delta, t}(z) - u(t, z)   &=  \delta^{\alpha_1\alpha}(U_0 - u(t, z))\\ &\quad + (1-\delta^{\alpha_1\alpha})(\rho_{\Tilde{s}(z)}(z)  +K(\Tilde{s}(z))^2-u(t, z)-c\log(\Tilde{s}(z)/\delta)).
\end{align*}
Since $\rho_{\Tilde{s}(z)}(z)+K(\Tilde{s}(z))^2\geq u(t, z)$ and $u(t, \cdot) \leq 0$, it follows from (\ref{parameter bound}) that
\begin{align*}
    c(1-\delta^{\alpha_1\alpha})\log(\Tilde{s}(z)/\delta) \geq \delta^{\alpha_1\alpha}(U_0(z) - C_2).
\end{align*}
As $c = \varepsilon_0 A^{-1}\delta^{\alpha_1\alpha} - KA^{-1}\delta^2$, by choosing $\delta \leq \min \{\delta_0, (\varepsilon_0/2K)^{1/(2-\alpha_1\alpha)}\}$, we get $c \geq \frac{1}{2}\varepsilon_0A^{-1}\delta^{\alpha_1\alpha}$ and therefore
\begin{align*}
    \Tilde{s}(z) \geq \delta \exp\left(\frac{2A(U_0(z)-C_2)}{\varepsilon_0(1-\delta_0^{\alpha_1\alpha})}\right).
\end{align*}
Let us denote 
\begin{align*}
    \eta(z) := \exp\left(\frac{2A(U_0(z) - C_2)}{\varepsilon_0(1-\delta_0^{\alpha_1\alpha})}\right).
\end{align*}
\indent Now, fix $z \in \rm{Amp}(\theta)$. Since $\Tilde{s}(z) \geq \delta\eta(z)$ and $s \mapsto \rho_su(t, \cdot)+Ks^2$ is increasing, we get
\begin{equation}\label{4.10}
    \begin{aligned}
        \rho_{\delta\eta(z)}u(t, z) - u(t, z) &\leq \rho_{\Tilde{s}(z)}u(t, z) + K(\Tilde{s}(z))^2-u(t, z) \\ 
    &\leq U_{c, \delta, t}(z) - u(t, z)\\
    &\leq \frac{1}{1-\delta^{\alpha_1\alpha}}(u_{\delta, t}(z) - \delta^{\alpha_1\alpha}U_0(z) - u(t, z)).
    \end{aligned}
\end{equation}
It follows from (\ref{parameter bound}) and (\ref{4.10}) that
\begin{align*}
    \rho_{\delta\eta(z)}u(t, z) - u(t, z) \leq (1-\delta_0^{\alpha_1\alpha})^{-1}\delta^{\alpha_1\alpha}(C_2-U_0).
\end{align*}
We now replace $\delta$ by $\delta /\eta(z)$. For $\delta/\eta(z) \leq \delta_0$, 
\begin{align*}
    \rho_{\delta}u(t, z) - u(t, z) &\leq (1-\delta_0^{\alpha_1\alpha})^{-1}\delta^{\alpha_1\alpha}(C_2-U_0(z)) (\eta(z))^{\alpha_1\alpha} \\
    &= (1-\delta_0^{\alpha_1\alpha})^{-1}(C_2-U_0(z))\exp\left(\frac{2A\alpha_1\alpha(U_0(z)-C_2)}{\varepsilon_0(1-\delta_0^{\alpha_1\alpha})}\right)\delta^{\alpha_1\alpha}.
\end{align*}
Since $U_0$ is locally bounded from below on $\rm{Amp}(\theta)$, it implies local H\"older continuity of $u(t, \cdot)$ on $\rm{Amp}(\theta)$. \\
\indent Finally, note that the H\"older exponent $\alpha_1\alpha$ does not depend on $t$. The joint continuity of $u$ on $X_T$ and the local H\"older continuity of $u(t, \cdot)$ on ${\rm{Amp}}(\theta)$ for all $t \in (0, T)$ follows from the local uniform Lipschitzness of $u$ in $(0, T)$.
\end{proof}
\section{Uniqueness Result}
In this section, we give a proof of Theorem \ref{intro_main result 2}. We proceed as in \cite[Section 3]{Ka25b}.
\subsection{Mixed type inequalities} We need the following lemmas about the mixed type inequalities. 
\begin{lemma}\label{mixed type inequality}
    Assume that $\varphi, \psi \in PSH(X, \theta) \cap L^{\infty}(X)$ satisfy
    \begin{align*}
        (\theta+dd^c\varphi)^n \geq f(\theta+dd^c\psi)^n
    \end{align*}
    for some $f \in L^1(X, (\theta+dd^c\psi)^n)$. Then, we have
    \begin{align*}
        (\theta+dd^c\varphi) \wedge (\theta+dd^c\psi)^{n-1} \geq f^{1/n}(\theta+dd^c\psi)^n.
    \end{align*}
\end{lemma}
\begin{proof}
    It follows from Dinew's mixed type inequality in \cite[Theorem 5.1]{Di09}.
\end{proof}
\begin{lemma}\label{mixed type comparison principle}
    Assume that $u, \varphi \in PSH(X, \theta) \cap L^{\infty}(X)$. Then
    \begin{align*}
        \int_{\{\varphi < u\}}(\theta+dd^cu) \wedge (\theta+dd^c\varphi)^{n-1}\leq \int_{\{\varphi < u\}}(\theta+dd^c\varphi)^n.
    \end{align*}
\end{lemma}
\begin{proof}
    The proof is almost same with the one for the classical comparison principle (see e.g. \cite[Proposition 9.2]{GZ17}).
\end{proof}
\subsection{Boundary behavior of supersolutons}
Let $v_0 \in PSH(X, \omega_0) \cap L^{\infty}(X)$. We show the following property of supersolutions.
\begin{lemma}\label{Cauchy boundary data of supersolution}
    Assume that there exist $\varphi \in PSH(X, \theta) \cap L^{\infty}(X)$ and a constant $C > 0$ satisfying
    \begin{align*}
        \begin{cases}
            &d\mu \leq C(\theta+dd^c\varphi)^n  \text{ in } X, \\
            &\sup_X \varphi = 0.
        \end{cases}
    \end{align*}
    Assume that $v \in \mathcal{P}(X_T, \omega_t) \cap L^{\infty}(X)$ satisfies
    \begin{align*}
        \begin{cases}
            & dt \wedge (\omega_t+dd^cv)^n \leq e^{\partial_tv+F(t, x, v)}dt \wedge d\mu \text{ in $X_T$}, \\
            &\lim_{t \rightarrow 0+}v(t, \cdot) = v_0 \text{ in } L^1(X, d\mu).
        \end{cases}
    \end{align*}
    Then there exist $t_0 > 0$ and a function $\eta(t) \in C([0, t_0])$ such that $\eta(0) = 0$ and
    \begin{align*}
        v(t, x) \geq v_0(x) - \eta(t)
    \end{align*}
    for $dt \wedge d\mu$-a.e. $(t, x) \in (0, t_0) \times X$.
\end{lemma}
\begin{proof}
    Let $t_0 := \min\{\frac{1}{2\lambda_F}, \frac{T}{2}\}$ and fix $0 < \varepsilon \leq t_0$. Recall that $\dot{\omega}_t \geq -A\omega_t$, which implies that 
    \begin{align*}
        \omega_{t+\varepsilon} \geq e^{-At}\omega_{\varepsilon}.
    \end{align*}
    Let us define a map $(t, x) \mapsto U^{\varepsilon}(t, x)$ on $(0, t_0) \times X$ by
    \begin{align*}
        U_{\varepsilon}(t, x) := \left(1-\frac{t}{t_0}\right)e^{-At}v(\varepsilon, x) + \frac{t}{t_0}(\varphi(x)-C_1) + t\left(n\log\left(\frac{t}{t_0}\right)-n-\log C\right)
    \end{align*}
    where
    \begin{align*}
        C_1 := t_0\lVert F(t, x, \lVert U_{\varepsilon}\rVert_{L^{\infty}((0, t_0) \times X)})\rVert _{L^{\infty}((0, t_0) \times X)}+(At_0+1)\lVert v\rVert_{L^{\infty}((0, t_0) \times X)}.
    \end{align*}
    For all $t \in (0, t_0)$,
    \begin{align*}
        (\omega_{t+\varepsilon}+dd^cU_{\varepsilon}(t, \cdot))^n &= \left(\left(1-\frac{t}{t_0}\right)(\omega_{t+\varepsilon}+e^{-At}dd^cv(\varepsilon, \cdot))+ \frac{t}{t_0} (\omega_{t+\varepsilon}+dd^c\varphi)\right)^n \\
        & \geq \left(1-\frac{t}{t_0}\right)^ne^{-Ant}(\omega_{\varepsilon}+dd^cv(\varepsilon, \cdot))^n + \left(\frac{t}{t_0}\right)^n(\theta+dd^c\varphi)^n\\
        & \geq \frac{t^n}{t_0^nC}d\mu
    \end{align*}
    and
    \begin{align*}
        &\partial_tU_{\varepsilon}+F(t, x, U_{\varepsilon})\\ &= e^{-At}v(\varepsilon, x)\left(\frac{t-1}{t_0}A-\frac{1}{t_0}\right)+(\varphi(x)-C_1)t_0^{-1}+n\log\left(\frac{t}{t_0}\right) - \log C + F(t, x, U_{\varepsilon})\\
        &\leq n\log\left(\frac{t}{t_0}\right)-\log C
    \end{align*}
    by our choice of $C_1$. Therefore we have
    \begin{align*}
        dt \wedge (\omega_{t+\varepsilon}+dd^cU_{\varepsilon})^n \geq e^{\partial_tU_{\varepsilon}+F(t, x, U_{\varepsilon})}dt \wedge d\mu \text{ in $(0, t_0) \times X$.}
    \end{align*}
    Moreover, $U_{\varepsilon}$ is locally uniformly semi-concave in $(0, t_0)$ and $U_{\varepsilon}(t, x) \rightarrow v(\varepsilon, x)$ uniformly in $X$ as $t \rightarrow 0+$. \\
    \indent Next, let us define a map $(t, x) \mapsto V_{\varepsilon}(t, x)$ on $(0, t_0) \times X$ by
    \begin{align*}
        V_{\varepsilon}(t, x) := v(t+\varepsilon, x) + C_2\varepsilon t
    \end{align*}
    for $C_2 > 0$. We claim that if $C_2$ is large enough, then
    \begin{equation}\label{3.2}
        \begin{aligned}
            dt \wedge (\omega_{t+\varepsilon}+dd^cV_{\varepsilon})^n &= dt \wedge (\omega_{t+\varepsilon}+dd^cv(t+\varepsilon, x))^n\\ &\leq e^{\partial_tv(t+\varepsilon, x) + F(t+\varepsilon, x, v(t+\varepsilon, x))} dt \wedge d\mu \\
        &\leq e^{\partial_tV_{\varepsilon}(t, x) + F(t, x, V_{\varepsilon}(t, x))}dt \wedge d\mu \text{ in $(0, t_0) \times X$.}
        \end{aligned}
    \end{equation}
    Indeed, let us fix $C_2 := 2\kappa_F$ where $\kappa_F$ satisfies
    \begin{align*}
        \lvert F(t_1, x, r) - F(t_2, x, r)\rvert \leq \kappa_F \lvert t_1-t_2\rvert
    \end{align*}
    for all $t_1, t_2 \in [0, t_0+\varepsilon)$ and $r \in [-\lVert v\rVert_{L^{\infty}(X_T)}, \lVert v\rVert_{L^{\infty}(X_T)}]$. It follows from the quasi-increasing property of $F$ and our choice of $t_0$ that
    \begin{align*}
        &\partial_tv(t+\varepsilon, x) + F(t+\varepsilon, x, v(t+\varepsilon, x))\\ &\quad \leq \partial_tv(t+\varepsilon, x) + F(t, x, v(t+\varepsilon, x)) + \kappa_F\varepsilon \\
        &\quad \leq \partial_tv(t+\varepsilon, x)+F(t, x, v(t+\varepsilon, x)+C_2\varepsilon t) + \kappa_F{\varepsilon}+ \lambda_F C_2\varepsilon t_0 \\
        &\quad \leq \partial_tV_{\varepsilon} + F(t, x, V_{\varepsilon}).
    \end{align*}
    This implies the last inequality of (\ref{3.2}). Note that $V_{\varepsilon}$ is locally uniformly semi-concave in $(0, T)$ and $V_{\varepsilon}(t, x) \rightarrow v(\varepsilon, x)$ uniformly in $X$ as $t \rightarrow 0+$. Therefore by Lemma \ref{first comparison principle}, we have
    \begin{align*}
        V_{\varepsilon}(t, x) \geq U_{\varepsilon}(t, x) \text{ in } (0, t_0) \times X. 
    \end{align*}
    This means that
    \begin{align*}
        &v(t+\varepsilon, x) \\
        &\quad \geq \left(1-\frac{t}{t_0}\right)e^{-At}v(\varepsilon, x) + \frac{t}{t_0}(\varphi(x)-C_1) + t\left(n\log\left(\frac{t}{t_0}\right)-n-\log C\right)-C_2\varepsilon t.
    \end{align*}
    Finally, since $v(t, \cdot) \rightarrow v_0$ in $L^1(d\mu)$ as $t \rightarrow 0+$, there exists a sequence $\{\varepsilon_k\}_{k \geq 1}$ such that $\varepsilon_k \rightarrow 0$ as $k \rightarrow \infty$ and $v(\varepsilon_{k}, x) \rightarrow v_0(x)$ as $k \rightarrow \infty$ for $\mu$-a.e. $x \in X$. By using this sequence, we obtain
    \begin{align*}
        &v(t, x) \\&\quad\geq \left(1-\frac{t}{t_0}\right)e^{-At}v_0(x) + \frac{t}{t_0} (\varphi(x)-C_1)+t\left(n\log \left(\frac{t}{t_0}\right)-n-\log C\right) \\
        &\quad\geq v_0(x) - \left(\left(1-\frac{t}{t_0}\right)e^{-At}-1\right) \lVert v_0\rVert_{L^{\infty}(X)}-\frac{t}{t_0}(\lVert \varphi\rVert_{L^{\infty}(X)}+C_1)\\
        &\quad \quad+t\left(n\log \left(\frac{t}{t_0}\right)-n-\log C\right)
    \end{align*}
    for $dt \wedge d\mu$-a.e. $(t, x) \in (0, t_0) \times X$.
\end{proof}
\subsection{General comparison principle} We now give a proof of Theorem \ref{intro_main result 2}. We first prove the following lemma, which will be a technical core of the proof of Theorem \ref{intro_main result 2}. For $u \in \mathcal{P}(X_T, \omega_t) \cap L^{\infty}(X_T)$, we define
\begin{align*}
    \partial_t^+u(t, x) := \lim_{\delta \rightarrow 0+}\frac{u(t+\delta, x) - u(t, x)}{\delta}
\end{align*}
and
\begin{align*}
    \partial_t^-u(t, x) := \lim_{\delta \rightarrow 0-}\frac{u(t+\delta, x) - u(t, x)}{\delta}.
\end{align*}
\begin{lemma}\label{first comparison principle}
    Assume that $u, v \in \mathcal{P}(X_T, \omega_t)\cap L^{\infty}(X_T)$ satisfy
    \begin{itemize}
        \item [(a)] for each $\varepsilon > 0$, there exists $\Tilde{\varepsilon} > 0$ such that
        \begin{align*}
            \int_{\{v+(t+1)\varepsilon < u\} \cap ((0, \Tilde{\varepsilon}) \times X)}dt \wedge d\mu = 0.
        \end{align*}
        \item [(b)] $u$ and $v$ are locally uniformly semi-concave in $(0, T)$,
        \item [(c)] for a.e. $t \in (0, T)$,
        \begin{align*}
            &(\omega_t+dd^cu(t, \cdot)) \wedge (\omega_t+dd^cv(t, \cdot))^{n-1} \\
            &\quad\geq \exp\left(\frac{\partial_tu(t, \cdot)-\partial_tv(t, \cdot)+F(t, \cdot, u)-F(t, \cdot, v)}{n}\right)(\omega_t+dd^cv(t, \cdot))^n,
        \end{align*}
        \item [(d)] $dt \wedge (\omega_t+dd^cv)^n \leq e^{\partial_tv+F(t, x, v)}dt \wedge d\mu$ in $X_T$,
        \item [(e)] $\partial_tu(t, x)$ is well-defined for all $(t, x) \in X_T$ and locally uniformly Lipschitz in $(0, T)$,
        \item [(f)] $F(t, z, r)$ is increasing in $r$.
    \end{itemize}
    Then $u \leq v$ in $X_T$.
\end{lemma}
\begin{proof}
    The proof is similar to the local case in  \cite[Lemma 3.4]{Ka25b}. We fix $\varepsilon > 0$ and define
    \begin{align*}
        E := \{v + (t+1)\varepsilon < u\}.
    \end{align*}
    Our goal is to show that $E = \emptyset$. Assume that it is not the case. Let us denote $E_t := \{x \in X ~\mid~ (t, x) \in E\}$. We define
    \begin{align*}
        t_1 := \inf\{t \in (0, T) ~\mid~ \mu(E_t) > 0\}.
    \end{align*}
    We first show that $t_1 > 0$. By the assumptions (a) and (d), there exists $\Tilde{\varepsilon} > 0$ such that
    \begin{align*}
        0 \leq \int_{E \cap ((0, \Tilde{\varepsilon}) \times X)}dt \wedge (\omega_t+dd^cv)^n \leq \int_{E \cap ((0, \Tilde{\varepsilon}) \times X}e^{\partial_tv+F(t, x, v)}dt \wedge d\mu = 0,
    \end{align*}
    which implies that
    \begin{align*}
        \int_{E \cap ((0, \Tilde{\varepsilon}) \times X)}dt \wedge (\omega_t+dd^cv)^n = 0.
    \end{align*}
    It follows from Lemma \ref{domination principle} that $E \cap ((0, \Tilde{\varepsilon}) \times X) = \emptyset$. Hence $\mu(E_t) = 0$ for all $0 < t < \Tilde{\varepsilon}$, which implies that $t_1 \geq \Tilde{\varepsilon} > 0$. \\
    \indent By replacing $\Tilde{\varepsilon}$ if necessary, we may assume that $t_1 + \frac{\Tilde{\varepsilon}}{2} < T$. It follows from the assumption (b) and (e) that there exists $C > 0$ such that
    \begin{align}\label{Lipschitzness of partial t u}
            \lvert \partial_tu(t, x) - \partial_tu(s, x)\rvert < C \lvert t-s\rvert
    \end{align}
    for all $t, s \in \left[t_1-\frac{\Tilde{\varepsilon}}{2}, t_1+\frac{\Tilde{\varepsilon}}{2}\right]$ and $x \in X$, and
    \begin{align}\label{concavity}
        u(t, x)-Ct^2,~ v(t, x)-Ct^2 \text{ are concave in } \left[t_1-\frac{\Tilde{\varepsilon}}{2}, t_1+\frac{\Tilde{\varepsilon}}{2}\right]
    \end{align}
    for all $x \in X$. Let us define $\delta := \min \left\{\frac{\Tilde{\varepsilon}}{4}, \frac{\varepsilon}{6C}\right\}$ and fix $t \in (t_1-\delta, t_1+\delta)$. Since $\mu(E_{t-\delta}) = 0$, we have
    \begin{align*}
        u(t, x) > v(t, x)+(t+1)\varepsilon \text{ for all } x \in E_t
    \end{align*}
    and
    \begin{align*}
        u(t-\delta, x) \leq v(t-\delta, x) + (t-\delta+1)\varepsilon \text{ for $\mu$-a.e. $x \in E_t$.}
    \end{align*}
    Therefore we have
    \begin{align}\label{comparing difference quotients}
        \frac{u(t, x) - u(t-\delta, x)}{\delta} > \frac{v(t, x) - v(t-\delta, x)}{\delta} + \varepsilon \text{ for $\mu$-a.e. $x \in E_t$.}
    \end{align}
    It follows from (\ref{concavity}) that 
    \begin{align*}
        \frac{(u(t, x) - Ct^2)-(u(t-\delta, x)-C(t-\delta)^2)}{\delta} &\leq \partial_t^+u(t-\delta, x) - 2C(t-\delta) \\
        &= \partial_tu(t-\delta, x) - 2C(t-\delta)
    \end{align*}
    and
    \begin{align*}
        \frac{(v(t, x)-Ct^2) - (v(t-\delta, x)-C(t-\delta)^2)}{\delta} \geq \partial_t^-v(t, x) - 2Ct
    \end{align*}
    for all $x \in E_t$. Therefore it follows from (\ref{comparing difference quotients}) that
    \begin{align}\label{partial t u(t-delta)}
        \partial_tu(t-\delta, x) + 2C\delta > \partial_t^-v(t, x) + \varepsilon \text{ for $\mu$-a.e. $x \in E_t$.}
    \end{align}
    Since $t$ and $t-\delta$ are in $\left(t_1-\frac{\Tilde{\varepsilon}}{2}, t_1+\frac{\Tilde{\varepsilon}}{2}\right)$, by combining (\ref{Lipschitzness of partial t u}) and (\ref{partial t u(t-delta)}), we get
    \begin{align}\label{comparing two partial derivatives}
        \partial_tu(t, x)  > \partial_tu(t-\delta, x) -C\delta>  \partial_t^-v(t, x) + \frac{\varepsilon}{2} \text{ for $\mu$-a.e. $x \in E_t$.}
    \end{align}
    \indent We now show that $E \cap ((t_1-\delta, t_1+\delta) \times X) = \emptyset$. For a.e. $t \in (t_1-\delta, t_1+\delta)$, we have
    \begin{equation}\label{main computation}
        \begin{aligned}
            &\int_{E(t, \cdot)}\exp\left(\frac{\varepsilon}{2n}\right)(\omega_t+dd^cv(t, \cdot))^n \\
            &\quad\leq \int_{E(t, \cdot)}\exp\left(\frac{\partial_tu(t, \cdot)-\partial_tv(t, \cdot)+F(t, \cdot, u)-F(t, \cdot, v)}{n}\right)(\omega_t+dd^cv(t, \cdot))^n \\
            &\quad\leq \int_{E(t, \cdot)}(\omega_t+dd^cu(t, \cdot)) \wedge (\omega_t+dd^cv(t, \cdot))^{n-1} \\
            &\quad \leq \int_{E(t, \cdot)}(\omega_t+dd^cv(t, \cdot))^n,
        \end{aligned}
    \end{equation}
    where we used (\ref{comparing two partial derivatives}) and the condition (f) for the first inequality and Lemma \ref{mixed type comparison principle} for the last inequality. By integrating both sides of (\ref{main computation}) with respect to $t$ for $(t_1-\delta, t_1+\delta)$, we get
    \begin{align*}
        0 \leq e^{\varepsilon/2n}\int_{t_1-\delta}^{t_1+\delta}dt \int_{E(t, \cdot)}(\omega_t+dd^cv(t, \cdot))^n \leq \int_{t_1-\delta}^{t_1+\delta}dt\int_{E(t, \cdot)}(\omega_t+dd^cv(t, \cdot))^n,
    \end{align*}
    which implies that 
    \begin{align*}
        \int_{t_1-\delta}^{t_1+\delta}dt \int_{E(t, \cdot)}(\omega_t+dd^cv(t, \cdot))^n = 0.
    \end{align*}
    It follows from Lemma \ref{fubini type lemma} that
    \begin{align*}
        \int_{E \cap ((t_1-\delta, t_1+\delta) \times X)}dt \wedge (\omega_t+dd^cv)^n = \int_{t_1-\delta}^{t_1+\delta}dt \int_{E(t, \cdot)}(\omega_t+dd^cv(t, \cdot))^n = 0.
    \end{align*}
    By using Lemma \ref{domination principle}, we get $E \cap ((t_1-\delta, t_1+\delta) \times X) = \emptyset$. Since it contradicts to the definition of $t_1$, we conclude that $E = \emptyset$.
\end{proof}
We next give a lemma which is about the condition (a) in Lemma \ref{first comparison principle}.
\begin{lemma}\label{lemma for the uniform control of Cauchy boundary data}
    Assume that $u, v \in \mathcal{P}(X_T, \omega_t) \cap L^{\infty}(X_T)$ satisfy
    \begin{itemize}
        \item [(a)] there exist $t_0 > 0$ and a function $\eta(t) \in C([0, t_0])$ such that $\eta(0) = 0$ and
        \begin{align*}
            v(t, x) \geq v_0(x) - \eta(t)
        \end{align*}
        for $dt \wedge d\mu$-a.e. $(t, x) \in (0, t_0) \times X$,
        \item [(b)] there exists $A > 0$ such that
        \begin{align*}
            \int_D u(t, \cdot)d\mu \leq \int_D u_0d\mu + At
        \end{align*}
        for all Borel subset $D \subset X$,
        \item [(c)] $u_0(x) \leq v_0(x)$ for all $x \in X$.
    \end{itemize}
    Then for each $\varepsilon > 0$, there exists $\Tilde{\varepsilon} > 0$ such that
    \begin{align*}
        \int_{\{v + (t+1)\varepsilon < u\} \cap ((0, \Tilde{\varepsilon}) \times X)}dt \wedge d\mu = 0.
    \end{align*}
\end{lemma}
\begin{proof}
    The proof is similar to the local case \cite[Lemma 3.6]{Ka25b}. We fix $\varepsilon > 0$ and define 
    \begin{align*}
        E := \{v+(t+1)\varepsilon < u\}.
    \end{align*}      
    Since we have
    \begin{align*}
        1 < \frac{u-v}{(t+1)\varepsilon}
    \end{align*}
    in $E$, for all $0 < \delta < t_0$, we get
    \begin{equation}\label{long inequality}
        \begin{aligned}
            &\int_{E \cap ((0, \delta) \times X)}dt \wedge d\mu \\
        &\quad\leq \int_{E \cap ((0, \delta) \times X)}\frac{u-v}{(t+1)\varepsilon}dt \wedge d\mu \\
        &\quad \leq \int_{E \cap ((0, \delta) \times X)} \frac{u-u_0}{(t+1)\varepsilon}dt \wedge d\mu + \int_{E \cap ((0, \delta) \times X)}\frac{v_0-v}{(t+1)\varepsilon}dt \wedge d\mu \\
        &\quad \leq \frac{A}{\varepsilon}\int_{E \cap ((0, \delta) \times X)}\frac{t}{t+1}dt\wedge d\mu + \frac{1}{\varepsilon}\int_{E \cap ((0, \delta) \times X)}\frac{\eta(t)}{t+1}dt \wedge d\mu \\
        &\quad \leq \frac{A\delta+\sup_{t \in (0, \delta)}\lvert \eta(t)\rvert}{\varepsilon}\int_{E \cap ((0, \delta) \times X)}dt \wedge d\mu.
        \end{aligned}
    \end{equation}
    We used the condition (c) for the second inequality and the condition (a) and (b) for the third inequality. Let us choose $0 < \Tilde{\varepsilon} < t_0$ small enough so that 
    \begin{align}\label{choosing tilde epsilon}
        A\Tilde{\varepsilon} \leq \frac{\varepsilon}{3} \text{ and } \sup_{t \in (0, \Tilde{\varepsilon})}\lvert \eta(t)\rvert \leq \frac{\varepsilon}{3}.
    \end{align}
    It follows from (\ref{long inequality}) and (\ref{choosing tilde epsilon}) that
    \begin{align*}
        0 \leq \int_{E \cap ((0, \Tilde{\varepsilon}) \times X)}dt \wedge d\mu \leq \frac{2}{3}\int_{E \cap ((0, \Tilde{\varepsilon}) \times X)}dt \wedge d\mu,
    \end{align*}
    which implies that
    \begin{align*}
        \int_{E \cap ((0, \Tilde{\varepsilon}) \times X)}dt \wedge d\mu = 0.
    \end{align*}
    Therefore we get the result.
\end{proof}
We now give a proof of Theorem \ref{intro_main result 2}. 
\begin{theorem}\label{main result_comparison principle}
    Let $u_0, v_0 \in PSH(X, \omega_0) \cap L^{\infty}(X)$. Assume that there exist $\varphi \in PSH(X, \theta) \cap L^{\infty}(X)$ and a constant $C > 0$ satisfying
    \begin{align*}
        d\mu \leq C(\theta+dd^c\varphi)^n \text{ in $X$.}
    \end{align*}
    Let $u, v \in \mathcal{P}(X_T, \omega_t) \cap L^{\infty}(X_T)$ satisfy $\lim_{t \rightarrow 0+}u(t, \cdot) = u_0$ and $\lim_{t \rightarrow 0+}v(t, \cdot) = v_0$ in $L^1(X, d\mu)$. Assume that
    \begin{itemize}
        \item [(a)] $u$ and $v$ are locally uniformly semi-concave in $(0, T)$,
        \item [(b)] $dt \wedge (\omega_t+dd^cu)^n \geq e^{\partial_tu+F(t, x, u)}dt \wedge d\mu$ in $X_T$,
        \item [(c)] $dt \wedge (\omega_t+dd^cv)^n \leq e^{\partial_tv+F(t, x, v)}dt \wedge d\mu$ in $X_T$.
    \end{itemize}
    If $u_0 \leq v_0$, then $u \leq v$.
\end{theorem}
\begin{proof}
    We may assume that $\sup_X \varphi = 0$. Let us fix $0 < T' < T$ and $\varepsilon_0 > 0$ such that $(1+\varepsilon_0)T' < T$. We first assume that there exists $B > 0$ satisfying
    \begin{align}\label{assumption on 1st derivative}
        t\lvert \partial_tu(t, x)\rvert \leq B
    \end{align}
    in $(0, T') \times X$. We will remove this assumption at the end of the proof. Next, we only consider the case when $F(t, z, r)$ is increasing in $r$. It suffices to consider only this case due to the invariance property of the equation explained in \cite[Section 3.3]{GLZ20}. By the assumption on $\omega_t$, there exists $A_1 > 0$ such that for all $t \in (0, T')$ and $s \in [1-\varepsilon_0, 1+\varepsilon_0]$,
    \begin{align*}
        \omega_t \geq (1-A_1\lvert s-1\rvert)\omega_{ts}.
    \end{align*}
    Let us define $\lambda_s := \frac{\lvert 1-s\rvert}{s}$ and $\alpha_s := s(1-\lambda_s)(1-A_1\lvert s-1\rvert) \in (0, 1)$. By shrinking $\varepsilon_0$ if necessary, we may assume that
    \begin{align*}
        \frac{\lambda_s}{1-\alpha_s} \geq \frac{1}{5+A_1} > 0.
    \end{align*}
    Let us define $\varepsilon_1 := (5+A_1)^{-1}$. By following the construction in \cite[Lemma 3.14]{GLZ20}, we obtain that there exists a constant $C_1 > 0$ such that for all $s \in [1-\varepsilon_0, 1+\varepsilon_0]$, 
    \begin{align*}
        W^s(t, x) := \frac{\alpha_s}{s}u(st, x) + (1-\alpha_s)\varepsilon_1\varphi - C_1\lvert s-1\rvert t - t\log C
    \end{align*}
    satisfies
    \begin{align}\label{property of Ws}
        \begin{cases}
            &W^s \in \mathcal{P}(X_{T'}, \omega_t) \cap L^{\infty}(X_{T'}), \\
            &dt \wedge (\omega_t+dd^cW^s)^n \geq e^{\partial_tW^s+F(t, x, W^s)}dt \wedge d\mu \text{ in } X_{T'}, \\
            &\lim_{t \rightarrow 0+}W^s(t, x) \leq u_0(x) \text{ for $\mu$-a.e. $x \in X$.} 
        \end{cases}
    \end{align}
    Indeed, fix $t \in (0, T')$ and $s \in [1-\varepsilon_0, 1+\varepsilon_0]$ such that
    \begin{align*}
        (\omega_{st}+dd^cu(st, \cdot))^n \geq e^{\partial_{\tau}u(st, \cdot)+F(st, \cdot, u)}d\mu \text{ in $X$.}
    \end{align*}
    By the choice of $W^s$, we have
    \begin{align*}
        &(\omega_t+dd^cW^s(t, \cdot))^n\\ &\geq ((1-\lambda_s)\omega_t+s^{-1}\alpha_sdd^cu(st, \cdot)+\lambda_s\omega_t+(1-\alpha_s)\varepsilon_1dd^c\varphi)^n \\
        &\geq (\alpha_ss^{-1}(\omega_{st}+dd^cu(st, \cdot))+(1-\alpha_s)\varepsilon_1(\theta+dd^c\varphi))^n \\
        &\geq \exp\left(\alpha_s(\partial_{\tau}u(st, \cdot)+F(st, \cdot, u)-n\log s)+(1-\alpha_s)(n\log\varepsilon-\log C)\right)d\mu.
    \end{align*}
    The last inequality holds by the convexity of the exponential function and Dinew's mixed type inequality \cite[Theorem 5.1]{Di09}. Next, 
    \begin{align*}
        \partial_tW^s(t, x)+F(t, x, W^s) = \alpha_s u(st, x)-C_1\lvert s-1\rvert - \log C+F(t, x, W^s).
    \end{align*}
    Since $\alpha_s-1 = O(\lvert s-1\rvert)$ and $F$ is bounded on $[0, (1+\varepsilon_0)T'] \times X \times [-(\lVert W^s\rVert_{L^{\infty}(X_T)}+\lVert u\rVert_{L^{\infty}(X_T)}), (\lVert W^s\rVert_{L^{\infty}(X_T)}+\lVert u\rVert_{L^{\infty}(X_T)})]$, there exists $C_1 > 0$ satisfying 
    \begin{align*}
        C_1\lvert s-1\rvert \geq \alpha_s n\log s+F(t, x, W^s) - \alpha_sF(st, x, u)+(\alpha_s-1)n\log \varepsilon_1
    \end{align*}
    for all $x \in X$, which implies that
    \begin{align}\label{inequality for Ws}
        (\omega_t+dd^cW^s(t, \cdot))^n \geq e^{\partial_tW^s(t, \cdot)+F(t, \cdot, W^s)}d\mu \text{ in $X$.}
    \end{align}
    \indent Now, we proceed as in the proof of the local case \cite[Theorem 3.8]{Ka25b}. Let $\eta_{\varepsilon}$ be the standard smooth mollifier on $\mathbb{R}$. We define
    \begin{align*}
        \Phi^{\varepsilon}(t, x) := \int_{\mathbb{R}}W^s(t, x)\eta_{\varepsilon}(s-1)ds - C_2\varepsilon(t+1)
    \end{align*}
    for some $C_2 > 0$ which will be chosen later.
    We first show that $\Phi^{\varepsilon}-O(\varepsilon)t$ and $v$ in $(0, T') \times X$ satisfy all the assumptions given in Lemma \ref{first comparison principle}. We choose $C_2 > 0$ such that 
    \begin{align*}
        \limsup_{t \rightarrow 0+}\Phi^{\varepsilon}(t, x) \leq u_0(x) \text{ for $\mu$-a.e. $x \in X$.}
    \end{align*}
    It follows from the definition of $\Phi^{\varepsilon}$ that $\partial_t\Phi^{\varepsilon}(t, x)$ is well-defined for all $(t, x) \in X_T$. Also, by the fact that $W^s$ is locally uniformly Lipschitz and locally uniformly semi-concave in $(0, T')$, $\partial_t\Phi^{\varepsilon}$ is locally uniformly Lipschitz in $(0, T')$ and $\Phi^{\varepsilon}$ is locally uniformly semi-concave in $(0, T')$.\\
    \indent We next show that for a.e. $t \in (0, T')$,
    \begin{align*}
        &(\omega_t+dd^c\Phi^{\varepsilon}(t, \cdot)) \wedge (\omega_t+dd^cv(t, \cdot))^{n-1} \\
        &\geq \exp\left(\frac{\partial_t\Phi^{\varepsilon}(t, \cdot)-\partial_tv(t, \cdot)+F(t, \cdot, \Phi^{\varepsilon})-F(t, \cdot, v)-O(\varepsilon)}{n}\right)(\omega_t+dd^cv(t, \cdot))^{n-1}.
    \end{align*}
    Since $W^s$ satisfies (\ref{property of Ws}), we have
    \begin{align*}
        (\omega_t+dd^cW^s(t, \cdot))^n &\geq e^{\partial_tW^s(t, \cdot)+F(t, \cdot, W^s)}d\mu \\
        &\geq e^{\partial_tW^s(t, \cdot)-\partial_tv(t, \cdot)+F(t, \cdot, W^s)-F(t, \cdot, v)}(\omega_t+dd^cv(t, \cdot))^n
    \end{align*}
    for a.e. $t \in (0, T)$. Fix such a $t$. It follows from Lemma \ref{mixed type inequality} that
    \begin{align*}
        &(\omega_t+dd^cW^s(t, \cdot)) \wedge (\omega_t+dd^cv(t, \cdot))^{n-1} \\
        &\quad \geq \exp\left(\frac{\partial_tW^s(t, \cdot)-\partial_tv(t, \cdot)+F(t, \cdot, W^s)-F(t, \cdot, v)}{n}\right)(\omega_t+dd^cv(t, \cdot))^n.
    \end{align*}
    Let $\chi(x)$ be a positive smooth test function on $X$. Let $\{U_k\}_{k = 1}^m$ be an open cover of $X$, and let $\rho_k(x)$ be a local potential of $\omega_t$ on $U_k$ for each $k$. Next, let $\{\chi_k(x)\}_{k = 1}^m$ be a partition of unity subordinate to $\{U_k\}_{k = 1}^m$. We have
    \begin{align*}
        &\int_{X}\chi(x)(\omega_t+dd^c\Phi^{\varepsilon}(t, \cdot)) \wedge (\omega_t+dd^cv(t, \cdot))^{n-1} \\ &\quad= \sum_{k = 1}^m\int_{U_k}\chi(x)\chi_k(x)(dd^c(\rho_k(x)+\Phi^{\varepsilon}(t, \cdot)) \wedge (\omega_t+dd^cv(t, \cdot))^{n-1} \\
        &\quad = \sum_{k = 1}^m\int_{U_k}(\rho_k(x)+\Phi^{\varepsilon}(t, \cdot))dd^c(\chi(x)\chi_k(x)) \wedge (\omega_t+dd^cv(t, \cdot))^{n-1} \\
        &\quad = \sum_{k = 1}^m\int_{U_k}\left(\int_{\mathbb{R}}(\rho_k(x)+W^s(t, \cdot)-\varepsilon(t+1))\eta_{\varepsilon}(s-1)ds\right) dd^c(\chi(x)\chi_k(x))\\& \quad\quad \quad \wedge (\omega_t+dd^cv(t, \cdot))^{n-1}.
    \end{align*}
    By Fubini's theorem, the last term becomes
    \begin{align*}
        &\sum_{k = 1}^m \int_{\mathbb{R}}\eta_{\varepsilon}(s-1)\left(\int_{U_k}(\rho_k+W^s(t, \cdot)-\varepsilon(t+1))dd^c(\chi\chi_k) \wedge (\omega_t+dd^cv(t, \cdot))^{n-1}\right)ds \\
        &\quad= \sum_{k = 1}^m\int_{\mathbb{R}}\eta_{\varepsilon}(s-1)\left(\int_{U_k}\chi\chi_k(\omega_t+dd^cW^s(t, \cdot)) \wedge (\omega_t+dd^cv(t, \cdot))^{n-1}\right)ds \\
        &\quad \geq \int_{\mathbb{R}}\eta_{\varepsilon}(s-1)\left(\int_X \chi e^{n^{-1}(\partial_t(W^s(t, \cdot)-v(t, \cdot))+F(t, \cdot, W^s)-F(t, \cdot, v))}(\omega_t+dd^cv(t, \cdot))^n\right)ds.
    \end{align*}
    By the Fubini's theorem again and Jensen's inequality, the last term becomes
    \begin{align*}
        &\int_{X}\chi\left(\int_{\mathbb{R}}\eta_{\varepsilon}(s-1)e^{n^{-1}(\partial_t(W^s(t, \cdot)-v(t, \cdot))+F(t, \cdot, W^s) - F(t, \cdot, v))}ds\right)(\omega_t+dd^cv(t, \cdot))^n \\
        &\quad\geq \int_X \chi e^{\int_{\mathbb{R}}n^{-1}(\partial_t(W^s(t, \cdot)-v(t, \cdot))+F(t, \cdot, W^s)-F(t, \cdot, v))\eta_{\varepsilon}(s-1)ds}(\omega_t+dd^cv(t, \cdot))^n \\
        &\quad \geq \int_{X}\chi e^{n^{-1}(\partial_t(\Phi^{\varepsilon}(t, \cdot)-v(t, \cdot))+F(t, \cdot, \Phi^{\varepsilon})-O(\varepsilon)-F(t, \cdot, v))}(\omega_t+dd^cv(t, \cdot))^n.
    \end{align*}
    The last inequality follows from
    \begin{align*}
        \int_{\mathbb{R}}F(t, \cdot, W^s)\eta_{\varepsilon}(s-1)ds \geq F(t, \cdot, \Phi^{\varepsilon})-O(\varepsilon),
    \end{align*}
     which is obtained by following the proof of \cite[Proposition 3.16]{GLZ20} and using the assumption (\ref{assumption on 1st derivative}). \\
    \indent It remains to show the assumption (a) in Lemma \ref{first comparison principle}. To show this, we prove that
    \begin{itemize}
        \item [(i)] there exist $0 < t_0 < T'$ and a function $\eta(t) \in C([0, t_0])$ such that $\eta(0) = 0$ and
        \begin{align*}
            v(t, x) \geq v_0(x) - \eta(t)
        \end{align*}
        for $dt \wedge d\mu$-a.e. $(t, x) \in (0, t_0) \times X$,
        \item [(ii)] for each $\varepsilon > 0$, there exists $A = A(\varepsilon) > 0$ such that
        \begin{align*}
            \int_{D}(\Phi^{\varepsilon}(t, \cdot)-O(\varepsilon)t)d\mu \leq \int_D u_0 d\mu + At
        \end{align*}
        for all Borel subset $D \subset X$,
    \end{itemize}
    and apply Lemma \ref{lemma for the uniform control of Cauchy boundary data}. Here (i) follows from Lemma \ref{Cauchy boundary data of supersolution}. To prove (ii), it suffices to show that there exists $A_0 > 0$ such that
    \begin{align*}
        \int_{D}(W^s(t, \cdot)-C_2\varepsilon(t+1))d\mu \leq \int_D u_0d\mu + A_0t
    \end{align*}
    for all Borel subset $D \subset X$.
    Indeed, fix a Borel subset $D \subset X$. Recall that $W^s$ satisfies
    \begin{align*}
        \begin{cases}
            &dt \wedge (\omega_t+dd^cW^s)^n \geq e^{\partial_tW^s+F(t, x, W^s)}dt \wedge d\mu, \\
            &\lim_{t\rightarrow 0+}W^s(t, x) \leq u_0(x)+C_2\varepsilon \text{ for $d\mu$-a.e.} x \in X.
        \end{cases}
    \end{align*}
    It follows from elliptic Chern-Levine-Nirenberg inequality that there exists a constant $C_3 > 0$ satisfying
    \begin{align*}
        \int_{D}e^{\partial_tW^s(t, \cdot)+F(t, \cdot, W^s)}d\mu \leq \int_X (\omega_t+dd^cW^s(t, \cdot))^n \leq C_3.
    \end{align*}
    Let $C_4 := -\inf_{X_T \times [-\lVert W^s\rVert_{L^{\infty}(X_{T'})}, \lVert W^s\rVert_{L^{\infty}(X_{T'})}]}$. By using the proof of \cite[Proposition 2.3]{GLZ20}, we get
    \begin{align*}
        \int_D\partial_tW^s(t, \cdot)d\mu \leq C_4\mu(X) + \mu(X) \log C_3 - \mu(X)\log \mu(X) \leq A_0
    \end{align*}
    for some uniform constant $A_0 > 0$. It implies that
    \begin{align*}
        \int_{D}(W^s(t, \cdot)-C_2\varepsilon (t+1))d\mu \leq \int_{D}u_0d\mu+A_0t
    \end{align*}
    Therefore it follows from Lemma \ref{lemma for the uniform control of Cauchy boundary data} that $\Phi^{\varepsilon}-O(\varepsilon)t$ and $v$ satisfy the assumption (a) in Lemma \ref{first comparison principle}. Hence by using Lemma \ref{first comparison principle}, we get
    \begin{align*}
        \Phi^{\varepsilon}-O(\varepsilon)t \leq v \text{ in $(0, T') \times X$.}
    \end{align*}
    Since $\Phi^{\varepsilon} \rightarrow u$ pointwisely as $\varepsilon \rightarrow 0$, we get
    \begin{align*}
        u \leq v \text{ in $(0, T') \times X$.}
    \end{align*}
    \indent Finally, we remove the assumption (\ref{assumption on 1st derivative}). Let $\kappa_F > 0$ be a constant satisfying
    \begin{align*}
        \lvert F(t_1, z, u) - F(t_2, z, u)\rvert \leq \kappa_F \lvert t_1-t_2\rvert
    \end{align*}
    for any $t_1, t_2 \in [0, T']$. Fix $0 < \varepsilon_1 < T-T'$ and let $U^{\varepsilon_1}(t, x) := u(t+\varepsilon_1, x) - \kappa_F \varepsilon_1 t$ and $V^{\varepsilon_1}(t, x) := v(t+\varepsilon_1, x) + \kappa_F \varepsilon_1 t$. Then for a.e. $t \in (0, T)$, $U^{\varepsilon_1}$ and $V^{\varepsilon_1}$ satisfy
    \begin{align*}
(\omega_{t+\varepsilon_1}+dd^cU^{\varepsilon_1}(t, \cdot))^n &= (\omega_{t+\varepsilon_1}+dd^cu(t+\varepsilon_1))^n \\
&\geq e^{\partial_tu(t+\varepsilon_1, \cdot)+F(t+\varepsilon_1, \cdot, u(t+\varepsilon_1, \cdot))}d\mu \\
&\geq e^{\partial_tU^{\varepsilon_1}(t, \cdot)+\kappa_F\varepsilon_1+F(t+\varepsilon_1, \cdot, U^{\varepsilon_1}(t, \cdot))}d\mu \\
&\geq e^{\partial_tU^{\varepsilon_1}(t, \cdot)+F(t, \cdot, U^{\varepsilon_1})}d\mu
    \end{align*}
and
\begin{align*}
    (\omega_{t+\varepsilon_1}+dd^cV^{\varepsilon_1}(t, \cdot))^n &= (\omega_{t+\varepsilon_1}+dd^cv(t+\varepsilon_1))^n \\
    &\leq e^{\partial_tv(t+\varepsilon_1, \cdot)+F(t+\varepsilon_1, \cdot, v(t+\varepsilon_1, \cdot))}d\mu \\
    &\leq e^{\partial_tV^{\varepsilon_1}(t, \cdot)-\kappa_F\varepsilon_1+F(t+\varepsilon_1, \cdot, V^{\varepsilon_1}(t, \cdot))}d\mu \\
    &\leq e^{\partial_tV^{\varepsilon_1}(t, \cdot)+F(t, \cdot,V^{\varepsilon_1})}d\mu.
\end{align*}
    Hence by replacing $u$ with $U^{\varepsilon_1}$ and $v$ with $V^{\varepsilon_1}$, we conclude that $U^{\varepsilon_1} \leq V^{\varepsilon_1}$ in $X_{T'}$. By taking $\varepsilon_1 \rightarrow 0$, we get the result.
\end{proof}
\bibliographystyle{abbrv}
\bibliography{ref.bib}
\end{document}